\documentclass[11pt]{article}
\usepackage[utf8]{inputenc}
\usepackage[T1]{fontenc}
\usepackage{comment}
\usepackage{authblk}
\usepackage{centernot}
\usepackage{relsize}
\usepackage{changepage}
\usepackage{mathtools}
\usepackage[english]{babel}
\usepackage{amsmath}
\usepackage{amsthm}
\usepackage{amsfonts}
\usepackage{amssymb}
\usepackage{graphicx}
\usepackage[usenames,dvipsnames]{xcolor}
\theoremstyle{plain}
\newtheorem{theorem}{Theorem}[section]
\newtheorem{corollary}[theorem]{Corollary}

\newtheorem{proposition}[theorem]{Proposition}
\newtheorem{lemma}[theorem]{Lemma}

\newtheorem{claim}[theorem]{Claim}

\theoremstyle{definition}
\newtheorem{definition}[theorem]{Definition}
\newtheorem{remark}[theorem]{Remark}

\usepackage{empheq}

\makeatother
\definecolor{lyucol}{rgb}{0.5, 0.1, 0.1}
\definecolor{gray}{rgb}{0.5, 0.5, 0.5}

\usepackage[left=2cm,right=2cm,top=2cm,bottom=2cm]{geometry}
\usepackage{amssymb}
\usepackage{hyperref}
\makeatletter

\newlength{\dhatheight}

\usepackage[normalem]{ulem}
\usepackage{dsfont}
\usepackage{accents}
\usepackage{verbatim}
\usepackage{ulem}
\usepackage{pgf,tikz,pgfplots}
\pgfplotsset{compat = 1.17}
\usepackage[all]{xy}

\newcommand{\eps}{\varepsilon}

\newcommand{\cD}{\mathcal{D}}
\renewcommand{\Pr}{\mathbb{P}}
\newcommand{\cV}{\mathcal{V}}
\newcommand{\cC}{\mathcal{C}}
\newcommand{\cA}{\mathcal{A}}
\newcommand{\cB}{\mathcal{B}}
\newcommand{\cE}{\mathcal{E}}
\newcommand{\cF}{\mathcal{F}}
\newcommand{\cG}{\mathcal{G}}
\newcommand{\cH}{\mathcal{H}}

\newcommand{\cT}{\mathcal{T}}

\newcommand{\cP}{\mathcal{P}}

\newcommand{\bo}{\mathbf{0}}

\usepackage{subfig}

\usepackage{enumerate}

\title{\scshape
  On the first and second largest components in the percolated Random Geometric Graph}

\author[1,2]{Lyuben Lichev}
\author[1,2]{Bas Lodewijks}
\author[1,2,4]{Dieter Mitsche\footnote{Bas Lodewijks and Dieter Mitsche are supported by grant GrHyDy ANR-20-CE40-0002. Dieter Mitsche is also supported by grant Fondecyt 1220174.}}
\author[2,3]{Bruno Schapira}
\affil[1]{Univ. Jean Monnet, Saint-Etienne, France}
\affil[2]{Institut Camille Jordan, Lyon and Saint-Etienne, France}
\affil[3]{Univ. Aix-Marseille, Marseille, France}
\affil[4]{Pontif\'icia Univ. Cat\'olica, IMC, Santiago, Chile}

\begin{document}

\maketitle
 
\begin{abstract}
The percolated random geometric graph $G_n(\lambda, p)$ has vertex set given by a Poisson Point Process in the square $[0,\sqrt{n}]^2$, and every pair of vertices at distance at most 1 independently forms an edge with probability $p$. For a fixed $p$, Penrose proved that there is a critical intensity $\lambda_c = \lambda_c(p)$ for the existence of a giant component in $G_n(\lambda, p)$. Our main result shows that for $\lambda > \lambda_c$, the size of the second-largest component is a.a.s.\ of order $(\log n)^2$. Moreover, we prove that the size of the largest component rescaled by $n$ converges almost surely to a constant, thereby strengthening results of Penrose.

We complement our study by showing a certain duality result between percolation thresholds associated to the Poisson intensity and the bond percolation of  $G(\lambda, p)$ (which is the infinite volume version of $G_n(\lambda,p)$). Moreover, we prove that for a large class of graphs converging in a suitable sense to $G(\lambda, 1)$, the corresponding critical percolation thresholds converge as well to the ones of $G(\lambda,1)$.
\end{abstract}
\noindent
MSC Class: 05C80, 60C05, 60D05, 60K35

\section{Introduction}

The theory of percolation was introduced by Broadbent and Hammersley~\cite{BH57} more than 60 years ago. In the most classical setting, one is given a subgraph of $\mathbb Z^d$ for some $d\ge 2$ where every edge appears with probability $p\in [0,1]$, independently of all other edges. While the model has a very simple definition, it undergoes a phase transition for the existence of an infinite component which is not completely understood to this day. 

Several years after Broadbent and Hammersley, Gilbert~\cite{Gil61} proposed a new mathematical model of wireless networks, which gave rise to the field of continuum percolation. His model, known as the \emph{Random Geometric Graph}, is defined as follows: given $\lambda, R > 0$, the vertices of the graph are given by a Poisson Point Process with intensity $\lambda$ in $\mathbb R^2$, and whose edges are given by the pairs of points at distance at most $R$. In fact, as mentioned in his paper, one of the two parameters $\lambda$ and $R$ may be put to 1 by a suitable homothety of the plane. Later, Meester and Roy~\cite{MR96} generalized Gilbert's model by connecting randomly and independently pairs of points in the Poisson process with a probability depending on their relative positions, thus introducing the random connection model. In this paper, we aim at studying a 
particular example of the random connection model    
obtained by performing standard Bernoulli bond percolation on top of Gilbert's model, often called \emph{soft random geometric graph} or \emph{percolated random geometric graph}.

\subsection{Formal setup and our results}

For $\lambda > 0$, denote by $\mathrm{Po}(\lambda)$ a Poisson Point Process in $\mathbb{R}^2$ of intensity $\lambda$. Then, we define $G(\lambda,1)$ as the random geometric graph with vertex set $\mathrm{Po}(\lambda)$ and edge set the set of pairs of vertices at Euclidean distance at most one. Given $p\in [0,1]$, we further define $G=G(\lambda,p)$ as the graph obtained from $G(\lambda,1)$ after Bernoulli bond percolation with probability $p$.

We will often consider $\mathrm{Po}(\lambda)$ conditioned on containing the origin, that is, we add artificially the origin to the point process (this construction is known in a more general setup under the name \emph{Palm theory}, see for instance \cite{Pen03}). Then, we shall denote by $\theta(\lambda,p)$ the probability that the connected component of the origin in $G$ is infinite. By classical considerations from ergodic theory (see e.g.~\cite{MR96}) one may deduce the existence of a deterministic threshold $\lambda_0\in [0, \infty]$ (in fact, a standard coupling argument with site percolation on $\mathbb Z^2$ (see \cite{MR96, Pen03}) shows that $\lambda_0\in (0,\infty)$) such that:
\begin{itemize}
    \item for all $\lambda > \lambda_0$, the graph $G(\lambda, 1)$ contains an infinite connected component almost surely, and in particular $\theta(\lambda, 1) > 0$;
    \item for all $\lambda < \lambda_0$, the graph $G(\lambda, 1)$ contains no infinite component almost surely, and in particular $\theta(\lambda, 1) = 0$. 
\end{itemize}
For every $p\in (0,1]$ we define then
\begin{equation}\label{eq:lambdacp}
\lambda_c(p) = \inf\{\lambda\in \mathbb R : \theta(\lambda, p) > 0\}.
\end{equation}

Moreover, for $n\ge 1$, we consider the restriction $G_n = G_n(\lambda, p)$ of $G$ to the square $[0,\sqrt n]^2$. Also, we denote by $L_1(G_n)$ the number of vertices in the largest connected component of $G_n$.

We say that a sequence of events $(\mathcal E_n)_{n\ge 0}$ holds \emph{asymptotically almost surely} (which we abbreviate by \emph{a.a.s.}) if $\mathbb P(\mathcal E_n)\to 1$ as $n\to \infty$. A sequence of random variables $(X_n)_{n\ge 0}$ 
is said to be a.a.s.\ of order $\Theta_{\lambda,p}(f_n)$  
if there exist positive constants $c$ and $C$, which might depend on $\lambda$ and $p$, such that $\mathbb P(cf_n\le X_n\le Cf_n)\to 1$ as $n\to \infty$.

Our main contribution is a sharp result on the size of the second largest component of $G(\lambda,p)$ restricted to a square of area $n$. It extends a recent result by Penrose~\cite{PenroseDraft} who proved that the size of the largest connected component of $G_n$ rescaled by $n$ converges in probability to $\lambda \theta(\lambda,p)$, while the size of the second-largest component divided by $n$ converges in probability to 0. More precisely, we have the following main theorem: 
\begin{theorem}\label{thm 1}
Fix $\lambda > 0$ and $p\in (0,1]$, such that $\lambda>\lambda_c(p)$. Then $(n^{-1} L_1(G_n))_{n\ge 1}$ converges almost surely to $\lambda\theta(\lambda, p)$. Moreover, the size of the second-largest component in $G_n$ is a.a.s.\ of order $\Theta_{\lambda,p}\big((\log n)^2\big)$.
\end{theorem}

 The proof of the almost sure convergence for the largest component is based on estimates for the probability of crossing large rectangles with a fixed length-to-width ratio. The main difficulty here is that two edges that intersect (geometrically) in an interior point may still be in different connected components because of the bond percolation. To solve this problem, we use the classical technique of sprinkling, which consists of revealing the  Poisson Point Process in two steps with the idea to locally connect two crossing edges present after the first step with positive probability.

The result on the size of the second-largest component is the most delicate part. Proving the lower bound is the easier part: it consists simply in observing that in the square $[0,\sqrt n]^2$, there is a.a.s.\ a subsquare of side length of order $\log n$ with no point at distance 1 from its boundary and that contains at least $c(\log n)^2$ points for some $c > 0$. The proof of the upper bound is more elaborate. The main idea is the following: we begin by proving that a.a.s.\ every point $x$ in the square $[0,\sqrt n]^2$ is surrounded by `many' cycles in the giant component of $G_n$ close to $x$. Hence, if the connected component of $x\in G_n$ has `large' Euclidean diameter, it necessarily intersects geometrically each of the above cycles. In this case, the argument of sprinkling that was used for the largest component does not work directly: although adding new points would help to connect $x$ to the giant component, it could also create new components with large diameter. We overcome this difficulty by using local sprinkling only rather than global sprinkling. This allows us to prove that components with large diameter (which thus cross many cycles) must be part of the giant component, and at the same time no new components with large diameter are created.

\begin{remark}[A generalization of the model.] In fact, a careful inspection of our proof shows that Theorem~\ref{thm 1} holds for the more general random connection model mentioned above. More precisely, fix a function $g:(0,1]\to (0,1]$. Then, the random graph $G(\lambda, g)$ may be defined as a subgraph of $G(\lambda, 1)$ in which every edge between two vertices $x$ and $y$ is retained with probability $g(\|x-y\|_2)$,
independently of all other edges. In particular, the graph $G(\lambda, p)$ is a special case of this generalized setting with $g$ constant equal to $p$. Then, by defining $\lambda_c(g)$ and $\theta(\lambda, g)$ as before, Theorem~\ref{thm 1} holds for this more general setup with almost the same proof. However, in order to keep the notation of the paper simple, we decided to present a proof in the simpler percolated random geometric graph model only.
\end{remark}

\begin{remark}[Stretched exponential decay in the supercritical regime.] The proof of Theorem~\ref{thm 1} easily implies that in the infinite volume setting, the origin is in a component of size at least $n$ without being part of the infinite component with probability at least $\exp(-c \sqrt{n})$ for some $c = c(\lambda, p) > 0$. For more details, see Remark~\ref{rem stretched exp decay}.
\end{remark}

\begin{remark}
We note that the question about the size of the second largest component of the graph $G_n$ in dimension more than two remains open. 
Unfortunately, our techniques do not shed light on this more general setting as we substantially use the properties of the planar embedding of $G_n$: Indeed, planarity allows to ``glue'' paths that intersect each other via sprinkling (see the proofs of Theorem~\ref{theo:giant} and Proposition~\ref{lem box crossing}), and to use duality in an auxiliary percolation on the $\mathbb Z^2$ lattice in the proof of Corollary~\ref{cor.Bercross}.
\end{remark}

Besides the study of the largest component sizes, a second motivation for studying $G(\lambda,p)$ is the following `duality' question. Recalling the definition of $\lambda_c(p)$ in~\eqref{eq:lambdacp}, for every $\lambda\in [\lambda_0, \infty)$, we define in a similar way 
$$p_c(\lambda) = \inf\{p\in [0,1] : \theta(\lambda, p) > 0\}.$$
It is natural to ask whether the two functions $\lambda_c(p)$ and $p_c(\lambda)$ are strictly monotone, continuous, and inverse of each other. We give an answer to this question in the following proposition: 
 
\begin{proposition}\label{thm critical}
For any $\lambda\in (\lambda_0, \infty)$ and $p\in (0, 1)$, one has $\lambda_c(p_c(\lambda)) = \lambda$ and $p_c(\lambda_c(p)) = p$, respectively. In particular, $\lambda_c$ and $p_c$ are inverse bijections, and hence continuous and strictly decreasing.
\end{proposition}

The most difficult part of this proposition is the equality $p_c(\lambda_c(p))=p$, which essentially follows from the results proved by Franceschetti, Penrose and Rosoman in~\cite{FPR11}. Thus, our contribution here is to prove the other equality, which is the easier part. For this, we rely on a classical bound by Hammersley~\cite{Ham61} stating that for any infinite locally finite graph, the bond percolation threshold is bounded from above by the site percolation threshold, which implies that $p_c$ is a strictly decreasing function of $\lambda$.

\vspace{0.2cm}

Finally, a third motivation for analyzing the graph $G(\lambda,p)$ is related to Schramm's locality conjecture, see~\cite{BNP11} and also~\cite{CMT23, DCT17,H20, MT17} for some recent progress. 
Suppose that we are given an integer $k\ge 1$ and a real $\lambda>0$. Then, remove all vertices of the random geometric graph $G(\lambda,1)$ which have degree larger than some constant $k$ together with all the edges emanating from them. This defines a subgraph of $G(\lambda,1)$ in which all vertices have degree bounded by $k$. Moreover, as $k\to \infty$, these random subgraphs (say rooted at the origin) converge locally in the Benjamini--Schramm sense to $G(\lambda,1)$ (see e.g.~\cite{AL07, BS01} for more on this notion of local convergence). Thus, we were initially aiming to understand whether the bond percolation thresholds on these subgraphs converge to $p_c(\lambda)$, the bond percolation threshold of $G(\lambda,1)$. Our first main result confirms a more general version of this statement. We remark that, while Schramm's locality conjecture was stated in terms of vertex-transitive graphs, we assume translation invariance of the distribution of our graphs. First, a random set embedded in $\mathbb R^2$ is said to be {\it locally finite} if for any bounded domain $\Omega$, the restriction of the set to $\Omega$ is almost surely finite.

\begin{definition}\label{def.graph} Let $K\ge 0$ be given. We say that a random graph 
with a locally finite vertex set embedded in $\mathbb R^2$ is a \emph{$K$-dependent graph of geometric type} if the following two conditions are satisfied: 
\begin{itemize}
    \item The law of the graph is invariant by any translation of $\mathbb R^2$. 
    \item For any two domains $\Omega_1$ and $\Omega_2$ at Euclidean distance at least $K$, the restrictions of the graph to $\Omega_1$ and $\Omega_2$ are independent. 
\end{itemize}
We also say that a graph is of geometric type if there exists $K\ge 0$ such that it is a $K$-dependent graph of geometric type. 
\end{definition}

In particular, each of the graphs $G(\lambda,p)$, or more generally any random connection model with a connection function having bounded support, is of geometric type. The example mentioned above where vertices with degree larger than a constant $k$ are removed is another example of a $2$-dependent graph of geometric type. In this setting, a convenient notion of local convergence is the following.

\begin{definition}\label{def.convergence}
A sequence $(\mathcal G_k)_{k\ge 1}$ of graphs of geometric type in the sense of Definition~\ref{def.graph} is said to {\it converge locally} to another graph $\mathcal G$ of geometric type if for any bounded domain $\Omega\subseteq\mathbb R^2$, the distribution of the restriction of $\mathcal G_k$ to $\Omega$ converges to the distribution of the restriction of $\mathcal G$
to $\Omega$ as a sequence of finite graphs. In other words, for any finite graph $H$, 
$$\lim_{k\to \infty} \ \mathbb P\big((\mathcal G_k)_{| \Omega} = H\big) = \mathbb P\big(\mathcal G_{| \Omega} = H\big).$$
\end{definition}

Examples of graph sequences of geometric type which converge to $G(\lambda,p)$ include:
\begin{itemize}
\item  $G(\lambda_k,p_k)$ for any sequence $(\lambda_k)_{k\ge 1}$ and $(p_k)_{k\ge 1}$ converging respectively to $\lambda$ and $p$;
\item  $\mathcal G_k$, obtained after removing all vertices in $G(\lambda, p)$ with degree larger than $k$;
\item  $\widehat{\mathcal G}_{\eps}$, obtained after removing all vertices in $G(\lambda, p)$ with another vertex at distance at most $\eps$ (in $\mathbb R^d$) from them.
\end{itemize}

A natural guess is that both the bond and the site percolation thresholds are continuous for this notion of convergence in the following sense. Fix $K\ge 0$ and any sequence $(\mathcal G_k)_{k\ge 1}$ of $K$-dependent graphs of geometric type converging locally to some other graph $\mathcal G$. Denote by $p_c(\mathcal G_k)$ and $s_c(\mathcal G_k)$ respectively the bond and the site percolation thresholds of $\mathcal G_k$ (and similarly for $p_c(\mathcal G)$ and $s_c(\mathcal G)$). Then, one should always have
$$\lim_{k\to \infty} p_c(\mathcal G_k) = p_c(\mathcal G) \quad \text{and}\quad \lim_{k\to\infty} s_c(\mathcal G_k) = s_c(\mathcal G).$$ 
We are able to prove one direction of this claim for bond percolation when the limiting graph is $G(\lambda,1)$. We do not address the question of site percolation here, which is a harder problem, see below for some comments on it.

\begin{proposition}\label{thm cor thm 1}
Fix $\lambda > \lambda_0$, $K\ge 0$, and a sequence of $K$-dependent graphs $(\mathcal G_k)_{k\ge 1}$ of geometric type which converges locally to $G(\lambda,1)$ in the sense of Definition~\ref{def.convergence}. Then, with the previous notation, one has  
$$\limsup_{k\to \infty} p_c(\mathcal G_k) \le  p_c(\lambda).$$
If in addition $\mathcal G_k$ is a subgraph of $G(\lambda,1)$ for any $k\ge 1$, then 
$$\lim_{k\to \infty} p_c(\mathcal G_k)  =   p_c(\lambda).$$
\end{proposition}

In particular this proposition applies to the example of the graph $\mathcal G_k$ obtained after removing from $G(\lambda,1)$ vertices with degree larger than $k$.

Moreover, we remark that the proof works for bond percolation and not for site percolation. The reason for this hides in the fact that, in general, comparison between bond and site percolation fails for graphs of unbounded maximum degree. In particular, our proof technique of sprinkling in new vertices, which might potentially have high degrees, breaks down for site percolation. This lack of symmetry is also apparent from (the statement of) Theorem~\ref{thm 1} and Remark~\ref{remark.GS}.

\vspace{0.2cm}

\paragraph{Plan of the paper.} In Section~\ref{sec 2} we introduce notation and several preliminary results. Then, Sections~\ref{sec 3} and~\ref{sec 5} are dedicated to the proofs of Proposition~\ref{thm critical} and Proposition~\ref{thm cor thm 1}. Finally, in Section~\ref{sec 4} we present the proof of Theorem~\ref{thm 1}.

\section{Notation and preliminaries}\label{sec 2}

\subsection{Notation}

Given a graph $\mathcal G$, we 
denote by $|\mathcal G|$ its number of vertices, which we also call its \emph{size}, and by $L_1(\cG)$ the size of its largest connected component. Given an event $\cA$, we denote by $\overline{\cA}$ the complementary event.

\vspace{0.1cm}

We use classical asymptotic notation: for sequences $(a_n)_{n\in\mathbb N}$ and $(b_n)_{n\in\mathbb N}$, we say that $a_n=O(b_n)$, $a_n=\Theta(b_n)$, and $a_n=\Omega(b_n)$ when there exist constants $0<c_1<c_2<\infty$ and $n_0\in\mathbb N$ such that for all $n\geq n_0$, $a_n\leq c_2b_n$, $c_1b_n\leq a_n\leq c_2b_n$, and $a_n\geq c_1 b_n$, respectively. Moreover, we say that $a_n=o(b_n)$ when $\lim_{n\to\infty} a_n/b_n=0$. By default the limit variable is $n$, and the constants associated to $O(\cdot)$, 
$\Theta(\cdot)$ and $\Omega(\cdot)$ are independent from the parameters of the problem; if this is not the case, the parameters influencing the constants will be given as lower indices, 
for example $O_p(\cdot)$ or $\Theta_{\lambda, p}(\cdot)$.

\vspace{0.1cm}

For any $x=(x_1,x_2)\in \mathbb R^2$ and $r \ge 0$, denote $\Lambda_x(r) := [x_1-r,x_1+r]\times [x_2-r,x_2+r]$, and simply write $\Lambda(r)$ when $x=\bo := (0,0)$ (the origin). Furthermore, for all $n\in \mathbb N$, we define 
$\Lambda_n:=[0,\sqrt{n}]^2$ and given a domain $\Omega\subseteq \mathbb R^2$, we set $\partial \Omega$ to be the boundary of $\Omega$. We also write $\|x\|$ for the Euclidean norm of $x\in \mathbb R^2$.

\vspace{0.1cm}

Given two domains $\Omega_1, \Omega_2\subseteq \mathbb R^2$ and a graph $\mathcal G$ with vertex set included in $\mathbb R^2$, 
we denote by $\Omega_1\xleftrightarrow{\mathcal G\,} \Omega_2$, or simply by $\Omega_1\leftrightarrow \Omega_2$ when $\mathcal G$ is clear from the context, 
the event that there is a path in $\mathcal G$ starting from a vertex in $\Omega_1$ and ending at a vertex in $\Omega_2$. Given a domain $\Omega\subseteq \mathbb R^2$, we shall also write $\Omega_1\xleftrightarrow{\Omega\;} \Omega_2$, the same event with the restriction that the path is contained 
in $\Omega$. Furthermore, when $\Omega_1$ (or $\Omega_2$) is reduced to a singleton $\{x\}$, we simply write $x$. 

\vspace{0.1cm}

Now, we introduce some standard notation from percolation theory. Given a fixed graph $\mathcal G=(V,E)$ (which in our case shall mostly be $\mathbb Z^2$ or a subgraph of it), 
a \emph{configuration} $\omega$ is an element of $\{0,1\}^{E}$. As usual, we often identify $\omega$ with the 
subgraph of $\mathcal G$ with vertex set $V$ and edge set $\{e:\omega_e=1\}$. Edges from this set are 
called \emph{open}, and the other edges of $\cG$ are called \emph{closed}. For $q\in [0,1]$, Bernoulli bond percolation with parameter $q$ on $\mathcal G$ is the product probability measure $\mathbb P_q$ for which every edge is open with probability $q$, independently of the other edges (we omit the reference to the base graph $\mathcal G$ in this notation, as it should always be clear from the context to which graph it applies).    
Note that by definition, when the base graph $\mathcal G$ is the random geometric graph $G(\lambda,1)$, the subgraph obtained after Bernoulli bond percolation with parameter $p$ is $G(\lambda,p)$. 
We shall either denote its distribution by $\mathbb P_{\lambda,p}$, when we want to emphasize which parameters we consider, or simply by $\mathbb P$, when they are clear from the context. 

\vspace{0.1cm}

Given a fixed graph $\mathcal G= (V,E)$, a configuration $\omega$ is smaller than $\omega'$ if for every $e \in E$ one has $\omega_e \le \omega'_e$. An event $\cA\subseteq \{0,1\}^E$ is increasing if whenever $\omega \in A$ and 
$\omega \le \omega'$, then also $\omega' \in A$, and an event $\cA$ is decreasing if $\overline{\cA}$ is increasing. 
Then, when considering the random geometric graph  
$G= G(\lambda,p)$, we say that 
an event $\mathcal A$ is increasing if whenever $\omega$ and $\omega'$ are two instances of $G$ such that $\omega$ is a subgraph of $\omega'$ and $\omega$ is in $\mathcal A$, then $\omega'$ is also in $\mathcal A$.

\subsection{Preliminaries}

The following well-known inequality shows that two increasing events are positively correlated. It applies to a wide variety of random models, and both Bernoulli percolation and the percolated random geometric graph model are among them, see e.g.~\cite{DC18, PenroseDraft}. 

\begin{lemma}[Harris inequality / FKG inequality]\label{lem harris}
For any $q\in [0,1]$ and any increasing events $\cA$ and $\cB$, one has  
$$\mathbb P_q(\cA \cap \cB) \ge \mathbb P_q(\cA) \mathbb P_q(\cB).$$
Similarly, for any $\lambda>0$, $p\in [0,1]$ and any two increasing events $\cA$ and $\cB$, 
$$\mathbb P_{\lambda,p}(\cA \cap \cB) \ge \mathbb P_{\lambda,p}(\cA) \mathbb P_{\lambda,p}(\cB).$$
\end{lemma}

A consequence of Lemma~\ref{lem harris} is the so-called \emph{square-root trick}: given $k$ increasing (or decreasing) events $(\cA_i)_{i=1}^k$, and any $q\in [0,1]$, 
\begin{equation}\label{eq square-root trick}
\max_{1\le i\le k}\mathbb P_q(\cA_i)\ge 1 - (1-\mathbb P_q(\cup_{i=1}^k \cA_i))^{1/k}, 
\end{equation}
and similarly with $\mathbb P_{\lambda,p}$ instead of $\mathbb P_q$. 

\vspace{0.2cm}

Now, we present a couple of results concerning Bernoulli bond percolation on $\mathbb Z^2$ which will be used several times in this work. First we need a result on dependent bond percolation. For any $k\in \mathbb N$, we say that a bond percolation on a graph is \emph{$k$-dependent} if the states of any two (families of) edges at graph distance larger than $k$ are independent.

\begin{theorem}[see Theorem 0.0 in \cite{LSS97}]\label{thm LSS}
For every $k\ge 1$ and $p\in [0,1)$, there is $q_0 = q_0(k, p) < 1$ such that every $k$--dependent bond percolation measure on $\mathbb Z^2$, satisfying that any edge is open with probability $q > q_0$, dominates Bernoulli bond percolation with parameter $p$. 
\end{theorem}
Next, we state a result providing exponential decay of correlations in the subcritical regime. 

\begin{theorem}[\cite{DC18}, Theorem 3.3]\label{DC theorem}
For every $q\in [0, 1/2)$, there exists $c_q > 0$ such that, for all $x\in \mathbb Z^2$,  
$$\mathbb P_q(\bo\leftrightarrow x) \le \exp(-c_q\|x\|).$$
\end{theorem}
The last result we shall need on Bernoulli percolation on $\mathbb Z^2$ provides some concentration for the size of the largest component in finite volume. 

\begin{theorem}[see \cite{PP96}, Theorem 4]\label{PP theorem}
Fix $n\ge 1$, $H_n = \mathbb Z^2\cap [0, \sqrt{n}]^2$ and $q\in (1/2, 1]$. Denote by $H_{n,q}$ the graph obtained from $H_n$ after Bernoulli bond percolation with parameter $q$. Then, for every $\eps > 0$, there exists $c = c(q, \eps) > 0$ such that $\mathbb P_q(L_1(H_{n,q}) \le (1-\eps) \mathbb E L_1(H_{n,q}))\le  \exp(-c n)$.
\end{theorem}

Since the vertices of our graph are given by a homogeneous Poisson Point Process, the next formula, known under the name \emph{Campbell-Mecke formula}, will be useful in our analysis. 
We refer to \cite{PenroseDraft} for this version of the theorem, which can simply be deduced from more standard versions by integrating first against the Bernoulli percolation measure, conditionally on the Poisson Point Process.  

\begin{lemma}[Campbell-Mecke formula]\label{campbell-mecke}
Fix $\lambda > 0$, $p\in [0,1]$, and let $G=G(\lambda,p)$. For any non-negative measurable function $f$ on the set of pairs $(x,\mathcal G)$, where $\mathcal G$ is a graph whose vertex set is a locally finite subset of $\mathbb R^2$
and $x$ is a vertex of $\mathcal G$, one has 
$$
\mathbb E\Big[\sum_{X\in \mathrm{Po}(\lambda)} f\big(X, G\big)\Big] = \lambda \int_{\mathbb R^2}
\mathbb E[ f(x, G_x)]\, \mathrm{d}x, 
$$
where $G_x$ is obtained by adding $x$ to the vertex set of $G$, and independently for each other vertex $y$ of $G$ at distance smaller than one from $x$, 
an edge is added between $x$ and $y$ with probability $p$. 
\end{lemma}

We continue with a particular case of a theorem by Franceschetti, Penrose and Rosoman~\cite{FPR11} on the critical intensity for the appearance of a giant component in $G(\lambda, p)$.

\begin{theorem}[see Theorem~2.3 in~\cite{FPR11}]\label{thm:FPR}
For every $p_1, p_2\in [0,1]$ such that $p_1 < p_2$, $\lambda_c(p_1) > \lambda_c(p_2)$.
\end{theorem}

Finally, we state a standard Chernoff-type inequality for Poisson random variables. 

\begin{lemma}\label{Chernoff poisson}
Let $X$ be a Poisson random variable with mean $\lambda > 0$. Then, for any $\varepsilon > 0$,
$$\mathbb P(|X-\lambda| \ge \varepsilon)\le  2\exp\left(-\frac{\varepsilon^2}{2(\lambda +\varepsilon)}\right).$$
\end{lemma}

\section{Proof of Theorem~\ref{thm critical}}\label{sec 3}

First, we fix $\lambda\in (\lambda_0, \infty)$ and show that $\lambda_c(p_c(\lambda)) = \lambda$. We argue by contradiction, assuming that $\lambda_c(p_c(\lambda)) \neq \lambda$. We distinguish two cases:
\begin{itemize}
    \item if $\lambda_c(p_c(\lambda)) > \lambda$, then there is $\eps > 0$ such that $\lambda+\eps<\lambda_c(p_c(\lambda))$ and thus  $\theta(\lambda+\eps, p_c(\lambda)) = 0$, and hence $p_c(\lambda+\eps)\ge p_c(\lambda)$. However, $p_c$ is a non-increasing function, so $p_c(\lambda+\eps) = p_c(\lambda)$;
    \item if $\lambda_c(p_c(\lambda)) < \lambda$, then there is $\eps > 0$ such that $\lambda-\eps>\lambda_c(p_c(\lambda))$ and thus     $\theta(\lambda-\eps, p_c(\lambda)) > 0$, and hence $p_c(\lambda-\eps)\le p_c(\lambda)$. However, $p_c$ is a non-increasing function, so $p_c(\lambda-\eps) = p_c(\lambda)$.
\end{itemize}
In both cases, our assumption leads to the existence of $\lambda_1, \lambda_2\in (\lambda_0, \infty)$ satisfying $\lambda_1 < \lambda_2$ and $p_c(\lambda_1) = p_c(\lambda_2)$. 

Now, set $p = p_c(\lambda_2)+\delta$ for some $\delta > 0$ to be chosen sufficiently small. It thus holds that $G(\lambda_2,p)$ is supercritical. On the one hand, by a classical result of Hammersley~\cite{Ham61} the bond percolation threshold of $G(\lambda_2, p)$, which is $p_c(\lambda_2)/p$, is dominated by its site percolation threshold, which is $\lambda_c(p)/\lambda_2$. Hence, $\tfrac{\lambda_2}{p} \le \tfrac{\lambda_c(p)}{p_c(\lambda_2)}$. Hence, we find that $p_c(\lambda_2)/p\leq \lambda_c(p)/\lambda_2$. On the other hand,  $\theta(\lambda_1, p) > 0$ since $p > p_c(\lambda_1)$, so $\lambda_c(p)\le \lambda_1$. We conclude that $\tfrac{\lambda_2}{p}\le \tfrac{\lambda_1}{p_c(\lambda_2)}$, which rewrites as $\tfrac{\lambda_2}{\lambda_1}\le \tfrac{p}{p_c(\lambda_2)} = 1 + \tfrac{\delta}{p_c(\lambda_2)}$. Choosing $\delta$ sufficiently small (and using that $p_c(\lambda_2) > 0$) leads to a contradiction, which proves the equality.

To deduce the equality $p_c(\lambda_c(p)) = p$ for any fixed $p\in (0,1)$, a similar reasoning provides two distinct $p_1, p_2\in (0,1)$ satisfying $\lambda_c(p_1) = \lambda_c(p_2)$, which contradicts the statement of Theorem~\ref{thm:FPR}.
\hfill $\square$

\begin{remark}\label{remark.GS}
For an infinite connected graph $H$ with site percolation threshold $s_c$, bond percolation threshold $p_c$ and maximum degree $D \ge 3$, the equality $p_c(s_c(p))=p$ holds along similar lines without using Theorem~\ref{thm:FPR} as a black box. We provide a sketch of the argument. Suppose for contradiction that $p_c(s_c(p)) \neq p$ for some $p$ as above. Then, as in the proof of Theorem~\ref{thm critical}, there exist $p_1 < p_2$ with $s_c(p_1)=s_c(p_2)\in (0,1)$. Set $s=s_c(p_2)+\delta$ for some $\delta\in (0, 1-s_c(p_2))$ to be chosen below, and consider the graph $H(s, p_2)$ obtained from $H$ after site percolation with parameter $s$ and bond percolation with parameter $p_2$. By a result of Grimmett and Stacey~\cite{GS} we have $s_c \le 1-(1-p_c)^{D-1}$, or equivalently $(1-p_c)^{D-1} \le 1-s_c$ (see also \cite{CS00} for stronger results in this direction). Moreover, note that $\theta(s, p_1) > 0$ since $s > s_c(p_1)$, and thus $p_c(s) \le p_1$. 
Hence, applying the above inequality from~\cite{GS} for $\widehat{H} = H(s, p_2)$  (which has maximum degree at most $D$, site percolation threshold $s_c(p_c)/s$ and bond percolation threshold $p_c(s)/p_2$), we infer 
$$
\left(1-\frac{p_1}{p_2}\right)^{D-1} \le \left(1-\frac{p_c(s)}{p_2}\right)^{D-1}= (1-b_c(\widehat H))^{D-1} \le 1-s_c(\widehat H) = \frac{s - s_c(p_c)}{s}=\frac{\delta}{s}.
$$
Hence, choosing $\delta$ sufficiently small (and using $s_c(p_2) > 0$), we have the desired contradiction, and thus $p_c(\lambda_c(p))=p$.
\end{remark}

\section{\texorpdfstring{Proof of Proposition~\ref{thm cor thm 1}}{}}\label{sec 5}

We first prove that Bernoulli bond percolation preserves local convergence. 
For completeness, we also include the case of site percolation as it may be of independent interest.

\begin{lemma}\label{lem.convergence}
Let $(\mathcal G_k)_{k\ge 1}$ be a sequence of graphs of geometric type converging locally to $\mathcal G$ in the sense of Definitions \ref{def.graph} and \ref{def.convergence}. For $p\in [0,1]$, consider the subgraphs $\mathcal G_k^b(p)$ and $\mathcal G_k^s(p)$ obtained after performing respectively Bernoulli bond and site percolation on $\mathcal G_k$, and similarly for $\mathcal G^b(p)$ and $\mathcal G^s(p)$. Then, these graphs are of geometric type and moreover $(\mathcal G_k^b(p))_{k\ge 1}$ and $(\mathcal G_k^s(p))_{k\ge 1}$ converge locally respectively to $\mathcal G^b(p)$ and $\mathcal G^s(p)$.   
\end{lemma}
\begin{proof}
The fact that a graph of geometric type remains in this class after bond or site percolation is immediate from the definition. Now, assume that $(\mathcal G_k)_{k\ge 1}$ converges locally to $\mathcal G$, and let $p\in [0,1]$ be fixed. Let us consider first the case of bond percolation, which is slightly easier. Fix $\Omega$ to be some bounded domain of $\mathbb R^2$, and let $H$ be some  finite graph. Note that if $H$ is obtained from another graph $H'$ after bond percolation, then $H'$ belongs to the finite set $\mathcal A(H)$ of finite graphs with the same vertex set as $H$ and containing $H$ as a subgraph. 
Therefore, denoting by $E(H)$ the set of edges of a graph $H$, and taking advantage of the fact that $\mathcal A(H)$ is a finite set, we get 
\begin{align*}
\lim_{k\to \infty} \mathbb P\big(\mathcal G^b_k(p)_{|\Omega} = H\big) &  = \lim_{k\to \infty} \sum_{H'\in \mathcal A(H)}
\mathbb P\big((\mathcal G_k)_{|\Omega} = H'\big) p^{|E(H)|}(1-p)^{|E(H')|-|E(H)|}\\
& =\sum_{H'\in \mathcal A(H)}
\mathbb P\big(\mathcal G_{|\Omega} = H'\big) p^{|E(H)|}(1-p)^{|E(H')|-|E(H)|}  =  \mathbb P\big(\mathcal G^b(p)_{|\Omega} = H\big).
\end{align*}
We now consider the case of site percolation. The only additional difficulty is that now the set $\mathcal B(H)$ of graphs which can give rise to a fixed graph $H$ after site percolation is infinite. However, denoting by $V(H)$ the set of vertices of a graph $H$, and using Fatou's lemma, we deduce 
\begin{align*}
\liminf_{k\to \infty} \mathbb P\big(\mathcal G^s_k(p)_{|\Omega} = H\big) &  \ge \sum_{H'\in\mathcal B(H) }
\liminf_{k\to \infty}  \mathbb P\big((\mathcal G_k)_{|\Omega} = H' \big) p^{|V(H)|}(1-p)^{|V(H')|-|V(H)|}  \\
& = \sum_{H'\in \mathcal B(H)}
\mathbb P\big(\mathcal G_{|\Omega} = H'\big) p^{|V(H)|}(1-p)^{|V(H')|-|V(H)|}  = \mathbb P\big(\mathcal G^s(p)_{|\Omega} = H\big).
\end{align*}
Since this holds for any finite graph $H$, by another application of Fatou's lemma, we infer 
\begin{align*}
1 = \liminf_{k\to \infty} \sum_{H \text{ finite}}\mathbb P\big(\mathcal G^s_k(p)_{|\Omega} = H\big) \ge\sum_{H \text{ finite}}
\liminf_{k\to \infty}  \mathbb P\big((\mathcal G_k)_{|\Omega} = H \big)\ge \sum_{H \text{ finite}} \mathbb P\big(\mathcal G^s(p)_{|\Omega} = H\big) = 1.
\end{align*}
\noindent
Hence, for any finite graph $H$, $\lim_{k\to \infty} \mathbb P\big(\mathcal G^s_k(p)_{|\Omega} = H\big)= \mathbb P\big(\mathcal G^s(p)_{|\Omega} = H\big)$, which concludes the proof of the lemma.
\end{proof}

Now, we prove Theorem~\ref{thm cor thm 1}. Let $\lambda>\lambda_0$, and $K\ge 0$ be given. Let also $(\mathcal G_k)_{k\ge 1}$ 
be a sequence of $K$-dependent graphs of geometric type converging to $G(\lambda,1)$. We need to prove that $\limsup_{k\to \infty} p_c(\mathcal G_k) \le p_c(\lambda)$. To this end, we fix some $p>p_c(\lambda)$ and show that, with the notation from Lemma~\ref{lem.convergence}, for all $k$ sufficiently large, $\mathcal G_k^b(p)$ contains an infinite connected component.

The proof relies on a finite-size criterion. Consider a tessellation $\cT$ of $\mathbb R^2$ into squares of side length $\sqrt{m}$, where $m$ is a constant to be chosen sufficiently large later. For each square $Q$ in $\cT$, consider a partition of the square into $4$ smaller squares of side length $\sqrt m/2$, say $\{Q_i\}_{1\le i\le 4}$. Then, for a random graph with vertex set embedded in $\mathbb R^2$, consider the event $\mathcal A_Q$ that the following holds:    
\begin{itemize}
\item The second-largest component in $Q$ has size at most $\frac{\lambda \theta(\lambda,p)}{10}m$.
\item The largest components in each of the squares $Q_i$, for $i\in \{1,2,3,4\}$, have size at least $\frac{\lambda \theta(\lambda,p)}{8}m$.
\end{itemize}
Also, for two squares $Q$ and $Q'$ of $\cT$ sharing a common edge, consider the square $Q''$ of side length $\sqrt m$ in-between $Q$ and $Q'$, which is the union of the two smaller squares of $Q$ touching $Q'$ together with the two smaller squares of $Q'$ touching $Q$. Then, define 
$$\mathcal A_{Q,Q'} := \mathcal A_Q \cap \mathcal A_{Q'} \cap \mathcal A_{Q''}.$$

On the one hand, on $\mathcal A_{Q,Q'}$, the two largest components of $Q$ and $Q'$ are connected by construction. On the other hand, by Theorem~\ref{thm 1} (or by the results of \cite{PenroseDraft}) we know that for every $\delta>0$ and every sufficiently large $m$, $\mathbb P_{\lambda,p}(\mathcal A_{Q,Q'})>1-\delta$. Now, with a slight abuse of notation, denote by $\mathbb P_{k,p}$ the distribution of $\mathcal G_k^b(p)$. Then, by definition of the local convergence, Lemma~\ref{lem.convergence} and the fact that $\mathcal A_{Q,Q'}$ is a local event one has that, for any sufficiently large $k$ and any two squares $Q$ and $Q'$ sharing a common edge, $\mathbb P_{k,p}(\mathcal A_{Q,Q'})\ge 1-2\delta$. On the other hand, considering the graph isomorphic to $\mathbb Z^2$ with vertex set the squares in $\cT$, and declaring the edge between two neighboring squares $Q$ and $Q'$ open if the event $\mathcal A_{Q,Q'}$ holds, defines a bond percolation process on $\mathbb Z^2$. For every $m\ge K$, by hypothesis, this bond percolation is 1-dependent for each of the graphs $\mathcal G_k^b(p)$. Therefore, by Theorem~\ref{thm LSS}, choosing $\delta$ small enough (and then $m$ large enough) allows us to make this graph supercritical for any sufficiently large $k$, and thus proving that all the graphs $\mathcal G_k^b(p)$ have an infinite connected component.

The last part of the theorem is immediate since the critical bond percolation is a decreasing function for the partial order given by the inclusion of graphs. This concludes the proof of Proposition~\ref{thm cor thm 1}. \qed

\section{\texorpdfstring{Proof of Theorem~\ref{thm 1}}{}}\label{sec 4}
From now on, $\lambda > 0$ and $p\in (0,1]$ are fixed parameters such that $\lambda>\lambda_c(p)$. This in particular ensures that  $G=G(\lambda,p)$ contains an infinite connected component almost surely. 

The proof of Theorem~\ref{thm 1} is divided into three parts. In Section~\ref{sec.giant} we prove that a.a.s.\ $G_n$ contains a connected component of size $\Omega_{\lambda, p}(n)$. Despite the fact that this result has already been proved very recently by Penrose~\cite{PenroseDraft}, its proof serves as a base for the second and third steps. In Section~\ref{sec.second} we prove that a.a.s.\ the second-largest component in $G_n$ has size $\Theta_{\lambda, p}((\log n)^2)$, and finally in Section~\ref{sec:convergence} we show that $n^{-1} L_1(G_n)$ converges almost surely to $\lambda \theta(\lambda,p)$ as $n\to \infty$.

We introduce some additional notation, which will be used throughout this section.  A \emph{horizontal crossing} of a rectangle $[a,b]\times [c,d]$ in $G$ is a path in $G$ with edges, embedded as straight segments in the plane, such that:
\begin{itemize}
    \item the first edge intersects the segment connecting $(a,c)$ and $(a,d)$;
    \item all edges in the path but the first and the last ones are included in the rectangle $[a,b]\times [c,d]$;
    \item the last edge intersects the segment connecting $(b,c)$ and $(b,d)$.
\end{itemize}
A \emph{vertical crossing} is defined analogously. We denote by $\cH([a,b]\times [c,d])$ (respectively $\cV([a,b]\times [c,d])$) the event ``a path in $G$ crosses the rectangle $[a,b]\times [c,d]$ horizontally (vertically, respectively)''. We also let $\cH(b,d)$ and $\cV(b,d)$ denote the events $\cH([0,b]\times [0,d])$ and $\cV([0,b]\times [0,d])$, respectively.

\subsection{The existence of a giant component}\label{sec.giant}
This section is devoted to the proof of the fact that, in the supercritical regime, a.a.s.\ the graph $G_n$ contains a component of linear size, which we state as a separate theorem.

\begin{theorem}\label{theo:giant}
Assume that $\lambda>0$ and $p\in (0,1]$ are such that 
$\lambda>\lambda_c(p)$. Then a.a.s.\ one has $L_1(G_n)= \Theta_{\lambda,p}(n)$. 
\end{theorem}

The proof of this theorem is divided into two main parts. Firstly, we show in Proposition~\ref{lem box crossing} that for any $\kappa > 0$ the event $\cH(\kappa R,R)$ holds a.a.s.\ as $R\to \infty$. The proof of this part is inspired by the proof of Corollary 4.2 in~\cite{ATT18} in the context of a closely related model but, since we perform bond percolation on top of the random geometric graph, two intersecting edges need not necessarily be part of the same connected 
component (as is the case for the random geometric graph without percolation). To circumvent this difficulty, we use a carefully designed sprinkling procedure, as was sketched in the introduction.

In the second part, we construct an auxiliary graph that dominates a supercritical Bernoulli bond percolation on $\mathbb Z^2$. Then, we rely 
on known results for supercritical bond percolation on finite boxes of $\mathbb Z^2$ and transfer them to the original graph by using that, by construction, any linear-sized subgraph of the auxiliary graph corresponds to a linear-sized subgraph of the original graph. 

We start with the first part. 

\begin{proposition}\label{lem box crossing}
Assume that $\lambda>0$ and $p\in (0,1]$ are such that 
$\lambda>\lambda_c(p)$. Then for every $\kappa > 0$, one has $\mathbb P_{\lambda,p}(\cH(\kappa R, R))\to 1$, as $R\to \infty$.
\end{proposition}
\begin{proof}
The proof has four steps. The first one is to prove the result 
for $\kappa = 1/3$. 
The second step uses a result of  Penrose~\cite{PenroseDraft} saying that 
with probability bounded away from zero, for some large enough constant $K$ and two squares of side length $K$ at distance $R$ from each other, there exist one vertex in the first square and another in the second square that are connected by a path that stays within a box of side length $4R$ containing its endpoints. Note that $K$ is a large but fixed constant, and that we let $R$ tend to infinity, so that $R\gg K$.
In a third step, we prove that with probability bounded away from zero there is a path
surrounding the box $\Lambda(R)$ inside $\Lambda(3R)$.
Finally, as a last step,
we prove that if the result holds for a given $\kappa > 0$, then it also holds
for all $\kappa'\in [\kappa, 2\kappa)$.

\vspace{0.2cm}

\noindent \underline{Step 1.} 
We show that $\Pr(\cH(R/3, R))\to 1$ as $R\to \infty$. Define the events
\begin{align*}
&\cA_R = \cH([R/6, R/2]\times [-R/2, R/2]), \quad \cB_R = \cH([-R/2, -R/6]\times [-R/2, R/2]),\\
&\cC_R = \cV([-R/2, R/2]\times [R/6, R/2]), \quad \cD_R = \cV([-R/2, R/2]\times [-R/2, R/6]).
\end{align*}

\begin{figure}
\centering
\begin{tikzpicture}[line cap=round,line join=round,x=1cm,y=1cm]
\clip(-5.6,-2.1) rectangle (1.02,4.1);
\draw [line width=1pt] (-5,4)-- (-5,-2);
\draw [line width=1pt] (-5,-2)-- (1,-2);
\draw [line width=1pt] (1,-2)-- (1,4);
\draw [line width=1pt] (1,4)-- (-5,4);

\draw [line width=1pt] (-3,2)-- (-3,0);

\draw [line width=1pt] (-3,0)-- (-1,0);
\draw [line width=0.2pt, dash pattern={on 4pt off 4pt}] (-5,0)-- (1,0);

\draw [line width=1pt] (-1,0)-- (-1,2);

\draw [line width=1pt] (-1,2)-- (-3,2);
\draw [line width=0.2pt, dash pattern={on 4pt off 4pt}] (1,2)-- (-5,2);

\draw[line width=1pt] (-0.999999844151115,1.1830241847380731) -- (-0.999999844151115,1.1830241847380731);
\draw[line width=1pt] (-0.999999844151115,1.1830241847380731) -- (-0.980449883500054,1.2651335395187024);
\draw[line width=1pt] (-0.980449883500054,1.2651335395187024) -- (-0.9608999228489931,1.3428711431597598);
\draw[line width=1pt] (-0.9608999228489931,1.3428711431597598) -- (-0.9413499621979321,1.4163548822072882);
\draw[line width=1pt] (-0.9413499621979321,1.4163548822072882) -- (-0.9218000015468711,1.4856850979166378);
\draw[line width=1pt] (-0.9218000015468711,1.4856850979166378) -- (-0.9022500408958102,1.5509461472578492);
\draw[line width=1pt] (-0.9022500408958102,1.5509461472578492) -- (-0.8827000802447492,1.6122078903946695);
\draw[line width=1pt] (-0.8827000802447492,1.6122078903946695) -- (-0.8631501195936883,1.6695271101372735);
\draw[line width=1pt] (-0.8631501195936883,1.6695271101372735) -- (-0.8436001589426273,1.722948868868751);
\draw[line width=1pt] (-0.8436001589426273,1.722948868868751) -- (-0.8240501982915663,1.772507808445427);
\draw[line width=1pt] (-0.8240501982915663,1.772507808445427) -- (-0.8045002376405054,1.8182293985710678);
\draw[line width=1pt] (-0.8045002376405054,1.8182293985710678) -- (-0.7849502769894444,1.860131139145027);
\draw[line width=1pt] (-0.7849502769894444,1.860131139145027) -- (-0.7654003163383835,1.8982237220843836);
\draw[line width=1pt] (-0.7654003163383835,1.8982237220843836) -- (-0.7458503556873225,1.9325121581201128);
\draw[line width=1pt] (-0.7458503556873225,1.9325121581201128) -- (-0.7263003950362615,1.962996874067339);
\draw[line width=1pt] (-0.7263003950362615,1.962996874067339) -- (-0.7067504343852006,1.9896747860697097);
\draw[line width=1pt] (-0.7067504343852006,1.9896747860697097) -- (-0.6872004737341396,2.0125403543179314);
\draw[line width=1pt] (-0.6872004737341396,2.0125403543179314) -- (-0.6676505130830787,2.031586624742503);
\draw[line width=1pt] (-0.6676505130830787,2.031586624742503) -- (-0.6481005524320177,2.046806263180689);
\draw[line width=1pt] (-0.6481005524320177,2.046806263180689) -- (-0.6285505917809567,2.0581925875177634);
\draw[line width=1pt] (-0.6285505917809567,2.0581925875177634) -- (-0.6090006311298958,2.0657406033025665);
\draw[line width=1pt] (-0.6090006311298958,2.0657406033025665) -- (-0.5894506704788348,2.0694480483374083);
\draw[line width=1pt] (-0.5894506704788348,2.0694480483374083) -- (-0.5699007098277739,2.0693164517423552);
\draw[line width=1pt] (-0.5699007098277739,2.0693164517423552) -- (-0.5503507491767129,2.065352212993941);
\draw[line width=1pt] (-0.5503507491767129,2.065352212993941) -- (-0.5308007885256519,2.057567706438344);
\draw[line width=1pt] (-0.5308007885256519,2.057567706438344) -- (-0.511250827874591,2.045982416779065);
\draw[line width=1pt] (-0.511250827874591,2.045982416779065) -- (-0.49170086722353,2.030624111039159);
\draw[line width=1pt] (-0.49170086722353,2.030624111039159) -- (-0.47215090657246905,2.011530052498066);
\draw[line width=1pt] (-0.47215090657246905,2.011530052498066) -- (-0.4526009459214081,1.9887482621030825);
\draw[line width=1pt] (-0.4526009459214081,1.9887482621030825) -- (-0.43305098527034713,1.9623388328555427);
\draw[line width=1pt] (-0.43305098527034713,1.9623388328555427) -- (-0.4135010246192862,1.9323753026717556);
\draw[line width=1pt] (-0.4135010246192862,1.9323753026717556) -- (-0.3939510639682252,1.8989460912187637);
\draw[line width=1pt] (-0.3939510639682252,1.8989460912187637) -- (-0.37440110331716425,1.862156006224993);
\draw[line width=1pt] (-0.37440110331716425,1.862156006224993) -- (-0.3548511426661033,1.8221278247658628);
\draw[line width=1pt] (-0.3548511426661033,1.8221278247658628) -- (-0.33530118201504233,1.7790039550244363);
\draw[line width=1pt] (-0.33530118201504233,1.7790039550244363) -- (-0.31575122136398137,1.7329481840271956);
\draw[line width=1pt] (-0.31575122136398137,1.7329481840271956) -- (-0.2962012607129204,1.6841475168550277);
\draw[line width=1pt] (-0.2962012607129204,1.6841475168550277) -- (-0.27665130006185945,1.6328141128295262);
\draw[line width=1pt] (-0.27665130006185945,1.6328141128295262) -- (-0.2571013394107985,1.5791873241747052);
\draw[line width=1pt] (-0.2571013394107985,1.5791873241747052) -- (-0.23755137875973753,1.5235358426542438);
\draw[line width=1pt] (-0.23755137875973753,1.5235358426542438) -- (-0.21800141810867657,1.4661599596843766);
\draw[line width=1pt] (-0.21800141810867657,1.4661599596843766) -- (-0.1984514574576156,1.4073939454225624);
\draw[line width=1pt] (-0.1984514574576156,1.4073939454225624) -- (-0.17890149680655465,1.3476085523320647);
\draw[line width=1pt] (-0.17890149680655465,1.3476085523320647) -- (-0.1593515361554937,1.2872136487225945);
\draw[line width=1pt] (-0.1593515361554937,1.2872136487225945) -- (-0.13980157550443273,1.2266609877671666);
\draw[line width=1pt] (-0.13980157550443273,1.2266609877671666) -- (-0.12025161485337175,1.166447117495339);
\draw[line width=1pt] (-0.12025161485337175,1.166447117495339) -- (-0.10070165420231078,1.1071164372630085);
\draw[line width=1pt] (-0.10070165420231078,1.1071164372630085) -- (-0.0811516935512498,1.0492644061989524);
\draw[line width=1pt] (-0.0811516935512498,1.0492644061989524) -- (-0.06160173290018883,0.9935409091283094);
\draw[line width=1pt] (-0.06160173290018883,0.9935409091283094) -- (-0.04205177224912786,0.9406537854732133);
\draw[line width=1pt] (-0.04205177224912786,0.9406537854732133) -- (-0.022501811598066884,0.8913725266307961);
\draw[line width=1pt] (-0.022501811598066884,0.8913725266307961) -- (-0.00295185094700591,0.8465321473287966);
\draw[line width=1pt] (-0.00295185094700591,0.8465321473287966) -- (0.016598109704055064,0.8070372364590137);
\draw[line width=1pt] (0.016598109704055064,0.8070372364590137) -- (0.03614807035511604,0.7738661928888726);
\draw[line width=1pt] (0.03614807035511604,0.7738661928888726) -- (0.05569803100617701,0.7480756517513605);
\draw[line width=1pt] (0.05569803100617701,0.7480756517513605) -- (0.07524799165723799,0.7308051067136319);
\draw[line width=1pt] (0.07524799165723799,0.7308051067136319) -- (0.09479795230829896,0.7232817337245627);
\draw[line width=1pt] (0.09479795230829896,0.7232817337245627) -- (0.11434791295935993,0.726825421741585);
\draw[line width=1pt] (0.11434791295935993,0.726825421741585) -- (0.1338978736104209,0.7428540159371115);
\draw[line width=1pt] (0.1338978736104209,0.7428540159371115) -- (0.1534478342614819,0.7728887788848988);
\draw[line width=1pt] (0.1534478342614819,0.7728887788848988) -- (0.17299779491254286,0.8185600752267088);
\draw[line width=1pt] (0.17299779491254286,0.8185600752267088) -- (0.19254775556360382,0.8816132853196263);
\draw[line width=1pt] (0.19254775556360382,0.8816132853196263) -- (0.21209771621466478,0.9639149533644327);
\draw[line width=1pt] (0.21209771621466478,0.9639149533644327) -- (0.23164767686572574,1.0674591755154283);
\draw[line width=1pt] (0.23164767686572574,1.0674591755154283) -- (0.2511976375167867,1.194374233472129);
\draw[line width=1pt] (0.2511976375167867,1.194374233472129) -- (0.27074759816784766,1.3469294790532695);
\draw[line width=1pt] (0.27074759816784766,1.3469294790532695) -- (0.2902975588189086,1.527542475253549);
\draw[line width=1pt] (0.2902975588189086,1.527542475253549) -- (0.3098475194699696,1.7387863992836272);
\draw[line width=1pt] (0.3098475194699696,1.7387863992836272) -- (0.32939748012103054,1.9833977130938054);
\draw[line width=1pt] (0.32939748012103054,1.9833977130938054) -- (0.3489474407720915,2.264284106881933);
\draw[line width=1pt] (0.3489474407720915,2.264284106881933) -- (0.36849740142315246,2.5845327210860427);
\draw[line width=1pt] (0.36849740142315246,2.5845327210860427) -- (0.3880473620742134,2.9474186523622556);
\draw[line width=1pt] (0.3880473620742134,2.9474186523622556) -- (0.4075973227252744,3.3564137490485173);
\draw[line width=1pt] (0.4075973227252744,3.3564137490485173) -- (0.42714728337633534,3.8151957016147287);
\draw[line width=1pt] (0.42714728337633534,3.8151957016147287) -- (0.4466972440273963,4.327657433599871);
\draw[line width=1pt] (0.4466972440273963,4.327657433599871) -- (0.46624720467845726,4.897916798536738);

\draw [fill=black] (-2,1) circle (0.8pt);
\draw [fill=black] (-2.15,1) node {$\bo$};

\draw [decorate,decoration={brace,amplitude=6pt},xshift=0pt,yshift=0pt]
(-3,0) -- (-3,2) node {};
\draw [fill=black] (-3.4,1) node {$\tfrac{R}{3}$};

\draw [decorate,decoration={brace,amplitude=6pt},xshift=0pt,yshift=0pt]
(-5,-2) -- (-5,4) node {};
\draw [fill=black] (-5.4,1) node {$R$};

\end{tikzpicture}
\caption{Connecting the box $\Lambda(R/6)$ to the box $\Lambda(R/2)$: in this case, the rectangles associated to the events $\cC_R$ and $\cD_R$ are separated from the rest of the $R \times R$ square by dashed lines, and the path traverses vertically the rectangle on top, so the event $\cC_R$ is realized (the rotated rectangles associated to the events $\cA_R$ and $\cB_R$ are not explicitly separated).}
\label{fig 1}
\end{figure}
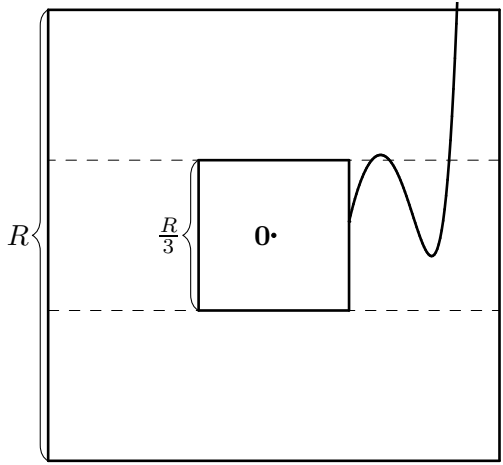

\noindent
Note that all four events correspond to crossing a rectangle with aspect ratio $3$ along its shorter side. Then, connecting the box $\Lambda(R/6)$ to the right (respectively the left, the top, or the bottom) side of $\Lambda(R/2)$ ensures that the event $\cA_R$ (respectively $\cB_R$, $\cC_R$, or $\cD_R$) is realized, see Figure~\ref{fig 1}.

Thus, the square-root trick~\eqref{eq square-root trick} and the fact that all four events have probability $\Pr(\cH(R/3,R))$ implies that
\begin{equation*}
\Pr(\cH(R/3,R))\ge 1 - (1-\Pr(\cA_R\cup \cB_R\cup \cC_R\cup \cD_R))^{1/4}\ge 1 - (1-\Pr(\Lambda(R/6)\leftrightarrow \mathbb R^2\setminus \Lambda(R/2)))^{1/4}.
\end{equation*}
Since by definition $G(\lambda,p)$ contains an infinite component almost surely, the probability that the box $\Lambda(R/6)$ intersects the infinite component tends to 1 as $R\to \infty$, which finishes the proof of Step 1.

\vspace{0.2cm}

\noindent \underline{Step 2.} For every sufficiently large constant $K>0$, one has
$$\liminf_{R\to \infty} \mathbb P\big(\Lambda(K)\xleftrightarrow{\Lambda(R)} \Lambda_{(R/2,0)}(K) \big) > 0.$$ 
This result is proved in \cite{PenroseDraft}, see the 
proof of Proposition 2.6 therein, and we refer to this paper for details.  
Let us stress that, in particular, both Steps 1 and 2 use the hypothesis $\lambda>\lambda_c(p)$. 

\vspace{0.2cm}

For the third step, we need a couple of new definitions. For $x\in \mathbb R^2$ and $r,R > 0$ satisfying
$0<r<R$, we define the annulus 
$$A_x(r,R) :=\Lambda_x(R)\setminus \Lambda_x(r),$$ 
and the event 
$$\mathrm{Circ}_x(r,R):=\{A_x(r,R) \text{ contains a cycle of }G\}.$$ 
We also simply write $A(r,R)$ and $\text{Circ}(r,R)$ for $A_{\bo}(r,R)$ and $\text{Circ}_{\bo}(r,R)$, respectively. In the next step, we prove that large annuli with constant ratio of their radii contain a cycle of $G$ with probability bounded away from zero.

\vspace{0.2cm}

\noindent \underline{Step 3.} We show that $\liminf_{R\to \infty} \Pr(\mathrm{Circ}(R, 3R)) > 0$.
Set $r=R/2$. For $i\in \{1,\dots,8\}$, let 
$$v_i = (ir-2R,2R), \quad v_{i+8}= (2R,2R-ir), \quad v_{i+16}=(2R-ir,-2R), \quad v_{i+24} = (-2R,-2R+ir).$$
Note that the points $(v_i)_{i=1}^{32}$ divide $\partial \Lambda(2R)$ into 32 equal segments of length $r$. By Step 2 we know that for every sufficiently large $K$ and $R = R(K)$ and for every $i\in \{1,\dots,32\}$, we have
\begin{equation}\label{eq incr events}
\mathbb P\left(\Lambda_{v_i}(K) \xleftrightarrow{\Lambda_{v_i}(R)} \Lambda_{v_{i+1}}(K)\right) \ge \frac 12,  
\end{equation}
where $v_{33} = v_1$, see Figure~\ref{fig:additional}.

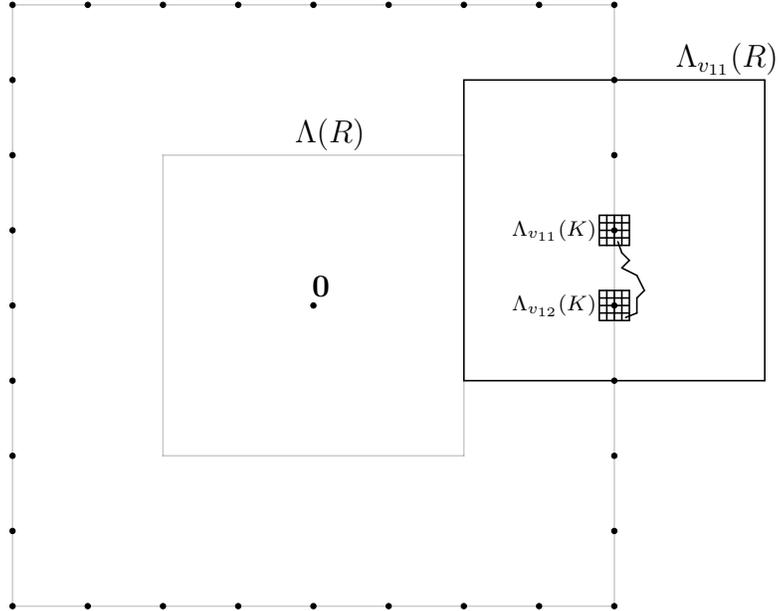
\begin{figure}
\centering
\begin{tikzpicture}[line cap=round,line join=round,x=1cm,y=1cm]
\clip(-7.092902004746363,-4.1) rectangle (12.42795927028891,4.1);
\draw [line width=0.6pt, opacity = 0.2] (0,2)-- (4,2);
\draw [line width=0.6pt, opacity = 0.2] (4,2)-- (4,-2);
\draw [line width=0.6pt, opacity = 0.2] (4,-2)-- (0,-2);
\draw [line width=0.6pt, opacity = 0.2] (0,2)-- (0,-2);
\draw [line width=0.6pt, opacity = 0.2] (-2,4)-- (6,4);
\draw [line width=0.6pt, opacity = 0.2] (6,4)-- (6,-4);
\draw [line width=0.6pt, opacity = 0.2] (6,-4)-- (-2,-4);
\draw [line width=0.6pt, opacity = 0.2] (-2,-4)-- (-2,4);
\draw [line width=0.6pt] (6,1.2)-- (6,0.8);
\draw [line width=0.6pt] (5.8,1.2)-- (5.8,0.8);
\draw [line width=0.6pt] (5.8,0.8)-- (6.2,0.8);
\draw [line width=0.6pt] (6.2,0.8)-- (6.2,1.2);
\draw [line width=0.6pt] (6.2,1.2)-- (5.8,1.2);
\draw [line width=0.6pt] (5.9,1.2)-- (5.9,0.8);
\draw [line width=0.6pt] (6.1,1.2)-- (6.1,0.8);
\draw [line width=0.6pt] (5.8,1.1)-- (6.2,1.1);
\draw [line width=0.6pt] (5.8,0.9)-- (6.2,0.9);
\draw [line width=0.6pt] (5.8,1)-- (6.2,1);
\draw [line width=0.6pt] (5.8,0.2)-- (6.2,0.2);
\draw [line width=0.6pt] (6,0.2)-- (6,-0.2);
\draw [line width=0.6pt] (6.2,0.2)-- (6.2,-0.2);
\draw [line width=0.6pt] (6.2,-0.2)-- (5.8,-0.2);
\draw [line width=0.6pt] (5.8,-0.2)-- (5.8,0.2);
\draw [line width=0.6pt] (5.9,0.2)-- (5.9,-0.2);
\draw [line width=0.6pt] (6.1,-0.2)-- (6.1,0.2);
\draw [line width=0.6pt] (5.8,0.1)-- (6.2,0.1);
\draw [line width=0.6pt] (5.8,0)-- (6.2,0);
\draw [line width=0.6pt] (5.8,-0.1)-- (6.2,-0.1);
\draw [line width=0.6pt] (6.047892216100419,0.8484229434135543)-- (6.1,0.7);
\draw [line width=0.6pt] (6.1,0.7)-- (6.2,0.6);
\draw [line width=0.6pt] (6.2,0.6)-- (6.1,0.5);
\draw [line width=0.6pt] (6.1,0.5)-- (6.3,0.4);
\draw [line width=0.6pt] (6.3,0.4)-- (6.4,0.2);
\draw [line width=0.6pt] (6.4,0.2)-- (6.3,0.1);
\draw [line width=0.6pt] (6.3,0.1)-- (6.3,-0.1);
\draw [line width=0.6pt] (6.3,-0.1)-- (6.150185636789366,-0.15935668263309644);
\draw [line width=0.6pt] (4,3)-- (8,3);
\draw [line width=0.6pt] (8,3)-- (8,-1);
\draw [line width=0.6pt] (8,-1)-- (4,-1);
\draw [line width=0.6pt] (4,-1)-- (4,3);
\begin{scriptsize}
\draw[color=black] (2.220388485764576,2.2723845445134128) node {\large{$\Lambda(R)$}};
\draw [fill=black] (-2,4) circle (1pt);
\draw [fill=black] (6,4) circle (1pt);
\draw [fill=black] (6,-4) circle (1pt);
\draw [fill=black] (-2,-4) circle (1pt);
\draw [fill=black] (-1,4) circle (1pt);
\draw [fill=black] (0,4) circle (1pt);
\draw [fill=black] (1,4) circle (1pt);
\draw [fill=black] (2,4) circle (1pt);
\draw [fill=black] (3,4) circle (1pt);
\draw [fill=black] (4,4) circle (1pt);
\draw [fill=black] (5,4) circle (1pt);
\draw [fill=black] (6,3) circle (1pt);
\draw [fill=black] (6,2) circle (1pt);
\draw [fill=black] (6,1) circle (1pt);
\draw [fill=black] (6,0) circle (1pt);
\draw [fill=black] (6,-1) circle (1pt);
\draw [fill=black] (6,-2) circle (1pt);
\draw [fill=black] (6,-3) circle (1pt);
\draw [fill=black] (5,-4) circle (1pt);
\draw [fill=black] (4,-4) circle (1pt);
\draw [fill=black] (3,-4) circle (1pt);
\draw [fill=black] (2,-4) circle (1pt);
\draw [fill=black] (1,-4) circle (1pt);
\draw [fill=black] (0,-4) circle (1pt);
\draw [fill=black] (-1,-4) circle (1pt);
\draw [fill=black] (-2,-3) circle (1pt);
\draw [fill=black] (-2,-2) circle (1pt);
\draw [fill=black] (-2,-1) circle (1pt);
\draw [fill=black] (-2,0) circle (1pt);
\draw [fill=black] (-2,1) circle (1pt);
\draw [fill=black] (-2,2) circle (1pt);
\draw [fill=black] (-2,3) circle (1pt);
\draw[color=black] (5.2,1) node {$\Lambda_{v_{11}}(K)$};
\draw[color=black] (5.2,0) node {$\Lambda_{v_{12}}(K)$};
\draw[color=black] (7.5,3.276092867928145) node {\large{$\Lambda_{v_{11}}(R)$}};
\draw [fill=black] (2,0) circle (1pt);
\draw[color=black] (2.1,0.25) node {\large{$\mathbf{0}$}};
\end{scriptsize}
\end{tikzpicture}
\caption{An illustration of the event described in~\eqref{eq incr events} for $i = 11$.}
\label{fig:additional}
\end{figure}

Since each of the events in \eqref{eq incr events} is increasing, we conclude by the FKG inequality (Lemma~\ref{lem harris}) that their intersection, which we call $\mathcal A$, has probability at least $2^{-32}$. Now, for every $i\in \{1,\dots,32\}$, fix a tessellation $\cT_i$ of the box $\Lambda_{v_i}(K)$ into squares of side length $1/\sqrt{5}$ (this is possible by choosing $K$ to be an integer multiple of $\sqrt{5}$) and let
$$\cB = \{Q\cap \mathrm{Po}(\lambda)\neq \emptyset, \ \text{for all }Q\in \cT_i \text{ and all } 1\le i\le 32\}.$$
\noindent
Since $\mathcal A$ and $\mathcal B$ are both increasing events, by the FKG inequality there is $\rho>0$ such that $\mathbb P(\mathcal A\cap \mathcal B)>\rho$, for all sufficiently large $R$. Also, choose $M > 0$ so that the event
\begin{equation*}
\cC = \{|\Lambda_{v_i}(K)\cap \mathrm{Po}(\lambda)|\le M, \text{ for all } 1\le i\le 32\}
\end{equation*}
holds with probability at least $1 - \rho/2$. Then, in particular, $\mathbb P(\mathcal A \cap \mathcal B\cap \mathcal C)\ge \rho/2$.

Now, observe that all three events $\mathcal A$, $\mathcal B$ and $\mathcal C$ are independent of the state of the edges inside the squares $(\Lambda_{v_i}(K))_{i=1}^{32}$. Indeed, $\cB$ and $\cC$ depend only on $\mathrm{Po}(\lambda)$, and $\mathcal A$ depends on $\mathrm{Po}(\lambda)$ and the state of the edges with at least one endpoint outside these squares. Thus, calling $\mathcal D$ the event that, for all $i\in \{1,\dots,32\}$, the vertices $\Lambda_{v_i}(K)\cap \mathrm{Po}(\lambda)$ induce a single connected component from $G(\lambda, p)$, we get $\mathbb P(\cD\mid \cA\cap \cB\cap \cC)\ge p^{32M^2}$. Hence, since $\mathcal A\cap \mathcal B\cap \mathcal C\cap \mathcal D\subseteq \mathrm{Circ}(R, 3R)$, we get
$$\mathbb P(\mathrm{Circ}(R, 3R))\ge \mathbb P(\mathcal A\cap \mathcal B\cap \mathcal C\cap \mathcal D)\ge \mathbb P(\mathcal A\cap \mathcal B\cap \mathcal C) \mathbb P(\mathcal D\mid \mathcal A\cap \mathcal B\cap \mathcal C)\ge \frac{\rho p^{32 M^2}}{2},$$
which finishes the proof of Step 3.

\vspace{0.2cm}

For the last step, we use a sprinkling argument. 

\vspace{0.2cm}

\noindent \underline{Step 4.} If $\mathbb P_{\lambda',p}(\mathcal H(\kappa R,R))\to 1$ as $R\to \infty$ for some fixed positive parameters $\kappa$, $\lambda'$ and $p$, then for any $\lambda>\lambda'$ one also has $\mathbb P_{\lambda,p}(\mathcal H(2\kappa R,R))\to 1$. 

\vspace{0.2cm}

Suppose that the statement holds for some  given $\kappa$, $\lambda'$ and $p$, and 
consider the rectangles 
\begin{equation*}
\Pi  = [0, 2\kappa R]\times [0,R], \quad 
\Pi_1  = [0, \kappa R]\times [0,R],  \quad  \Pi_2  = [\kappa R, 2\kappa R]\times [0,R].
\end{equation*}
In particular, $\Pi_1$ and $\Pi_2$ are respectively the left and right halves of $\Pi$. Then, divide the segment between the points $(\kappa R, 0)$ and $(\kappa R, R)$ into $N$ segments $I_1,\dots,I_N$, of equal length, where $N\in \mathbb N$ will be chosen appropriately later. Also, for any $i\in \{1,\dots,N\}$ and $j\in \{1,2\}$, let $\cB_{i,j}$ be the event that $\Pi_j$ contains a horizontal crossing which intersects $I_i$. For any fixed $N$, using our assumption and the square-root trick~\eqref{eq square-root trick},
we have for both $j\in \{1,2\}$ that
\begin{equation*}
\max_{1\le i\le N}\mathbb P_{\lambda',p}(\cB_{i,j})\ge 1 - (1 - \mathbb P_{\lambda',p}(\cH(\kappa R, R)))^{1/N},
\end{equation*}
which tends to 1 as $R\to \infty$. Denote by $i_0$ the smallest index realizing the maximum above (it is the same for both $j\in \{1,2\}$ by symmetry of $\Pi_1$ and $\Pi_2$), and let $x_0$ be the midpoint of the interval $I_{i_0}$.

By Step 3, there exist $R_0$ and $\delta > 0$ such that for every $R\ge R_0$, $\mathbb P_{\lambda',p}(\mathrm{Circ}(R, 3R))\ge \delta$. Assume now that $R/N \ge R_0$ and define 
$$K := \big\lfloor\log_4\big(\min(\kappa,1/2) N\big)\big\rfloor-1,\quad \text{and}\quad D:=R/N.$$ 
For every $k\in \{0,\dots,K\}$, define further the annulus $A_k := A_{x_0}(4^k D, 3\cdot 4^k D)$ and note that, by construction, any horizontal crossing of $\Pi_1$ or $\Pi_2$ that intersects $I_{i_0}$ also crosses each of these annuli. 
Denote by $\cD_k$ the event that there is 
a cycle inside the annulus $A_k$. Recall that by Step 3 and our choice of constants, each of these events has probability at least $\delta>0$ under $\mathbb P_{\lambda',p}$. Let $\cD:=\{\sum_{k=0}^K \mathds{1}_{\cD_k} \ge \delta K/2\}$ denote the event that in at least $\delta K/2$ many annuli $A_k$ there exists a cycle. Since $(\mathds{1}_{\cD_k})_{k=0}^K$ dominates an independent family of $K+1$ Bernoulli random variables with parameter $\delta$, Chernoff's inequality implies that $\cD$ holds with probability at least $1-\exp(-\Omega(\delta K))$. In particular, for every $\eps > 0$ and for every sufficiently large $R$ (allowing for $N$ and $K$ to be sufficiently large as well), the event $\cB_{i_0, 1}\cap \cB_{i_0, 2}\cap \cD$ holds with probability at least $1-\eps$.

Now, we first sample the graph $G(\lambda', p)$, and perform a sprinkling to connect the crossings of $\Pi_1$ and $\Pi_2$ intersecting $I_{i_0}$ to one of the cycles around $x_0$ on the event $\cB_{i_0, 1}\cap \cB_{i_0, 2}\cap \cD$. For this, fix some $\lambda>\lambda'$ and recall that by some well-known properties of Poisson Point Processes one may construct the graph $G(\lambda,p)$ by first sampling independent copies of 
$G(\lambda',p)$ and $G(\lambda-\lambda',p)$, and then adding independently an edge between any pair of vertices $v\in G(\lambda',p)$ and $v'\in G(\lambda-\lambda',p)$ satisfying $\|v-v'\|\le 1$ with probability $p$.

\begin{figure}
\centering
\begin{tikzpicture}[scale=1.3,line cap=round,line join=round,x=1cm,y=1cm]
\clip(-5,-6.198255010134753) rectangle (4.729938892772945,-0.7159507485854189);
\draw [line width=1.2pt] (4,-5)-- (-4,-5);
\draw [line width=0.4pt] (0,-6.198255010134753) -- (0,-0.7159507485854189);
\draw [line width=0.8pt] (-0.3,-3.2)-- (0.3,-3.2);
\draw [line width=0.8pt] (0.3,-3.2)-- (0.3,-3.8);
\draw [line width=0.8pt] (0.3,-3.8)-- (-0.3,-3.8);
\draw [line width=0.8pt] (-0.3,-3.8)-- (-0.3,-3.2);
\draw [line width=0.8pt] (-0.8,-2.7)-- (0.8,-2.7);
\draw [line width=0.8pt] (0.8,-2.7)-- (0.8,-4.3);
\draw [line width=0.8pt] (0.8,-4.3)-- (-0.8,-4.3);
\draw [line width=0.8pt] (-0.8,-4.3)-- (-0.8,-2.7);
\draw [line width=0.8pt] (-1,-2.5)-- (1,-2.5);
\draw [line width=0.8pt] (1,-2.5)-- (1,-4.5);
\draw [line width=0.8pt] (1,-4.5)-- (-1,-4.5);
\draw [line width=0.8pt] (-1,-4.5)-- (-1,-2.5);
\draw [line width=0.8pt] (-2.5,-6)-- (2.5,-6);
\draw [line width=0.8pt] (-2.5,-6)-- (-2.5,-1);
\draw [line width=0.8pt] (-2.5,-1)-- (2.5,-1);
\draw [line width=0.8pt] (2.5,-1)-- (2.5,-6);

\draw [line width=1.6pt] (-0.522936319660858,-3.1804983700246305)-- (-0.3497788519411222,-2.916639371594555);
\draw [line width=1.6pt] (-0.3497788519411222,-2.916639371594555)-- (-0.10241104091292813,-3.081551245613352);
\draw [line width=1.6pt] (-0.10241104091292813,-3.081551245613352)-- (0.18618473861996493,-3.0733056519124125);
\draw [line width=1.6pt] (0.18618473861996493,-3.0733056519124125)-- (0.4665349244519182,-3.023832089706773);
\draw [line width=1.6pt] (0.4665349244519182,-3.023832089706773)-- (0.4170613622462794,-3.378392618847187);
\draw [line width=1.6pt] (0.4170613622462794,-3.378392618847187)-- (0.6067100173678949,-3.5927780550716237);
\draw [line width=1.6pt] (0.6067100173678949,-3.5927780550716237)-- (0.5077628929566173,-3.881373834604519);
\draw [line width=1.6pt] (0.5077628929566173,-3.881373834604519)-- (0.26864067562936295,-4.087513677128015);
\draw [line width=1.6pt] (0.26864067562936295,-4.087513677128015)-- (0.0542552394049281,-3.939092990511098);
\draw [line width=1.6pt] (0.0542552394049281,-3.939092990511098)-- (-0.23434054012796496,-4.095759270828955);
\draw [line width=1.6pt] (-0.23434054012796496,-4.095759270828955)-- (-0.3910068204458212,-3.7824267101932407);
\draw [line width=1.6pt] (-0.3910068204458212,-3.7824267101932407)-- (-0.6218834440721356,-3.576286867669744);
\draw [line width=1.6pt] (-0.6218834440721356,-3.576286867669744)-- (-0.522936319660858,-3.1804983700246305);
\draw [line width=1.6pt] (-1.7515297811008885,-5.126458483446439)-- (-1.8257401244093467,-4.664705236193806);
\draw [line width=1.6pt] (-1.8257401244093467,-4.664705236193806)-- (-1.520653157474574,-4.145232833034594);
\draw [line width=1.6pt] (-1.520653157474574,-4.145232833034594)-- (-1.7432841873999487,-3.8071634912960604);
\draw [line width=1.6pt] (-1.7432841873999487,-3.8071634912960604)-- (-1.4299516267642363,-3.5927780550716237);
\draw [line width=1.6pt] (-1.4299516267642363,-3.5927780550716237)-- (-1.7845121559046477,-3.155761588921811);
\draw [line width=1.6pt] (-1.7845121559046477,-3.155761588921811)-- (-2.1060903102413,-2.9496217463983143);
\draw [line width=1.6pt] (-2.1060903102413,-2.9496217463983143)-- (-1.7845121559046477,-2.545587655052261);
\draw [line width=1.6pt] (-1.7845121559046477,-2.545587655052261)-- (-1.7103018125961895,-2.1168167826033883);
\draw [line width=1.6pt] (-1.7103018125961895,-2.1168167826033883)-- (-1.1825838157360422,-2.1662903448090276);
\draw [line width=1.6pt] (-1.1825838157360422,-2.1662903448090276)-- (-0.7950409117918714,-1.960150502285531);
\draw [line width=1.6pt] (-0.7950409117918714,-1.960150502285531)-- (-0.3085508834364232,-2.149799157407148);
\draw [line width=1.6pt] (-0.3085508834364232,-2.149799157407148)-- (-0.16837579052044654,-1.8117298156686135);
\draw [line width=1.6pt] (-0.16837579052044654,-1.8117298156686135)-- (0.6232012047697745,-2.075588814098689);
\draw [line width=1.6pt] (0.6232012047697745,-2.075588814098689)-- (0.8293410472932695,-1.9683960959864708);
\draw [line width=1.6pt] (0.8293410472932695,-1.9683960959864708)-- (1.3158310756487177,-2.355938999930644);
\draw [line width=1.6pt] (1.3158310756487177,-2.355938999930644)-- (1.7363563543966476,-2.4219037495381635);
\draw [line width=1.6pt] (1.7363563543966476,-2.4219037495381635)-- (1.6044268551816108,-2.9331305589964347);
\draw [line width=1.6pt] (1.6044268551816108,-2.9331305589964347)-- (1.8600402599107446,-3.155761588921811);
\draw [line width=1.6pt] (1.8600402599107446,-3.155761588921811)-- (1.678637198490069,-3.5927780550716237);
\draw [line width=1.6pt] (1.678637198490069,-3.5927780550716237)-- (1.8353034788079252,-3.881373834604519);
\draw [line width=1.6pt] (1.8353034788079252,-3.881373834604519)-- (1.5879356677797312,-4.268916738548693);
\draw [line width=1.6pt] (1.5879356677797312,-4.268916738548693)-- (1.8188122914060456,-4.45031979996937);
\draw [line width=1.6pt] (1.8188122914060456,-4.45031979996937)-- (1.5879356677797312,-4.7718979543060245);
\draw [line width=1.6pt] (1.5879356677797312,-4.7718979543060245)-- (1.266357513443079,-4.763652360605085);
\draw [line width=1.6pt] (1.266357513443079,-4.763652360605085)-- (1.0849544520224033,-5.1511952645492585);
\draw [line width=0.8pt, opacity=0.2] (1.0849544520224033,-5.1511952645492585)-- (0.8128498598913899,-5.225405607857717);
\draw [line width=0.8pt, opacity=0.2] (0.8128498598913899,-5.225405607857717)-- (0.7468851102838715,-5.62943969920377);
\draw [line width=0.8pt, opacity=0.2] (0.7468851102838715,-5.62943969920377)-- (0.38407898744252017,-5.7118956362131685);
\draw [line width=0.8pt, opacity=0.2] (0.38407898744252017,-5.7118956362131685)-- (0.16144795751714552,-5.316107138568055);
\draw [line width=0.8pt, opacity=0.2] (0.16144795751714552,-5.316107138568055)-- (-0.23434054012796496,-5.283124763764296);
\draw [line width=0.8pt, opacity=0.2] (-0.23434054012796496,-5.283124763764296)-- (-0.6301290377730754,-5.678913261409409);
\draw [line width=0.8pt, opacity=0.2] (-0.6301290377730754,-5.678913261409409)-- (-0.9434615984087878,-5.217160014156777);
\draw [line width=0.8pt, opacity=0.2] (-0.9434615984087878,-5.217160014156777)-- (-1.0836366913247646,-4.804880329109784);
\draw [line width=0.8pt, opacity=0.2] (-1.0836366913247646,-4.804880329109784)-- (-1.4711795952689353,-4.796634735408844);
\draw [line width=0.8pt, opacity=0.2] (-1.4711795952689353,-4.796634735408844)-- (-1.7515297811008885,-5.126458483446439);
\begin{scriptsize}
\draw [fill=black] (0,-3.5) circle (0.8pt);
\draw [fill=black] (-1.7515297811008885,-5.126458483446439) circle (0.8pt);
\draw [fill=black] (1.0849544520224033,-5.1511952645492585) circle (0.8pt);

\draw[color=black] (0.6,-2.9) node {\large{$\gamma_{k_\ell}'$}};
\draw[color=black] (2,-2.9) node {\large{$\gamma_{k_{\ell+1}}'$}};

\draw[color=black] (0.15,-3.6) node {\scriptsize{$x_0$}};

\draw[color=black] (-0.55,-4.1) node {\large{$A_{k_\ell}$}};
\draw[color=black] (-2.15,-5.8) node {\large{$A_{k_{\ell+1}}$}};

\draw[color=black] (-3.5,-4.75) node {\large{$\Pi$}};

\end{scriptsize}
\end{tikzpicture}
\caption{In the figure, the central vertex is $x_0$; around it, there are two cycles, $\gamma_{k_\ell}$ and $\gamma_{k_{\ell+1}}$, positioned in the annuli $A_{k_\ell}$ and $A_{k_{\ell+1}}$. Note that $\gamma_{k_\ell}' = \gamma_{k_\ell}$ but $\gamma_{k_{\ell+1}}'\subsetneq \gamma_{k_{\ell+1}}$ ($\gamma_{k_\ell}'$ and $\gamma_{k_{\ell+1}}'$ are thickened while $\gamma_{k_{\ell+1}}\setminus \gamma_{k_{\ell+1}}'$ is transparent). The proportions of the side lengths of the boxes, centered at $x_0$, are not the real ones in the above schematic representation.}
\label{fig 5}
\end{figure}
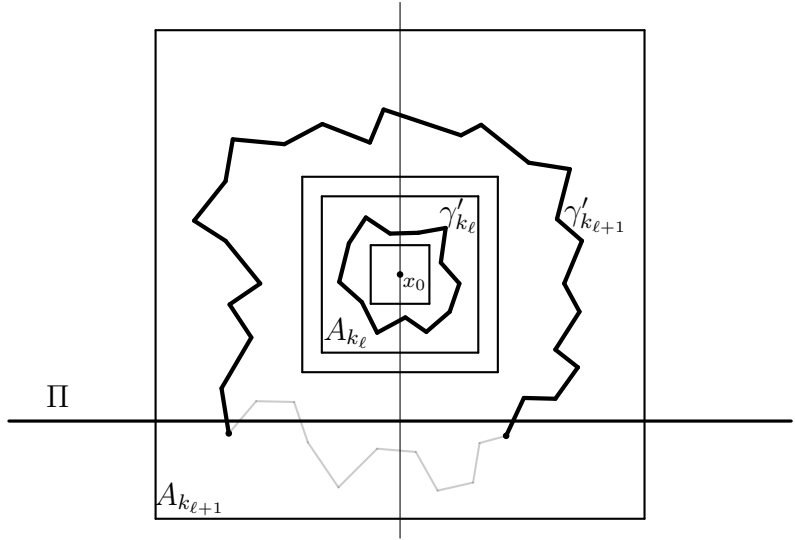

So, assume that the event $\cB_{i_0, 1}\cap \cB_{i_0, 2}\cap \cD$ holds for $G(\lambda',p)$, and let $P_1$ and $P_2$ be horizontal crossings of $\Pi_1$ and $\Pi_2$, respectively. Also, set $L = \lfloor\delta K/2\rfloor$ and let $(\gamma_{k_{\ell}})_{\ell=1}^L$ be cycles of $G(\lambda', p)$ provided by $\cD$ in the annuli $(A_{k_{\ell}})_{\ell=1}^L$, respectively. Note that parts of these paths may be outside $\Pi$. Thus, for every $\ell\in \{1,\dots,L\}$, denote by $\gamma_{k_{\ell}}'\subseteq \gamma_{k_{\ell}}$ a path or a cycle whose vertices, except possibly the first and the last one, are in $\Pi$ and which separates $x_0$ from the left and the right side of $\Pi$, see Figure~\ref{fig 5}. Then, connecting the paths $P_1$ and $P_2$ to the same path or cycle among $(\gamma_{k_{\ell}}')_{\ell=1}^L$ within $\Pi$ forms a horizontal crossing of $\Pi$.

Now, for each $j\in \{1,2\}$ and each $\ell\in \{1,\dots,L\}$, there are two edges, one in $P_j$ and one in $\gamma_{k_{\ell}}'$, that intersect. By the triangle inequality, we claim that there is one endvertex of the first edge and one endvertex of the second edge which both lie in $\Pi_j$ and are at distance at most $3/2$ from each other: indeed, letting $u_{j,\ell}$ and $v_{j,\ell}$ be the two endvertices of the first edge, $w_{j,\ell}$ be one of the endvertices of the second edge that lies in $\Pi_j$, and $z$ be the intersection point of the two edges, 
$$\min\{\|u_{j,\ell} - w_{j,\ell}\|, \|w_{j,\ell} - v_{j,\ell}\|\}\le \frac{\|u_{j,\ell} - w_{j,\ell}\| + \|w_{j,\ell} - v_{j,\ell}\|}{2}\le \frac{\|u_{j,\ell}-z\| + \|z-v_{j,\ell}\| +  2\|z-w_{j,\ell}\|}{2} \le \frac{3}{2}.$$ 

Fix two vertices as above, and note that the area of the region, obtained by intersecting the two unit balls centered around them and $\Pi_j$, is at least the area of the intersection of the unit balls around $\bo$, around $(3/2, 0)$ and the quarter-plane $\{(x,y): x,y \ge 0\}$, which we denote by $\alpha$. We conclude that, for every $\ell \in \{1,\dots,L\}$, $P_1$ and $P_2$ connect to $\gamma_{k_{\ell}}$ after adding $G(\lambda-\lambda',p)$, with probability at least $((1-\mathrm{e}^{-(\lambda-\lambda')\alpha})p^2)^2$ and these events are all independent of each other since the annuli $(A_k)_{k=0}^K$ are disjoint. Thus, for every sufficiently large $R$, the probability that $G(\lambda, p)$ contains a horizontal crossing of $\Pi$ is at least
$$(1-\eps)(1-(1-(1-\mathrm{e}^{-(\lambda-\lambda')\alpha})^2p^4)^L)\ge 1 - 2\eps.$$
Since this statement holds for any $\eps\in (0,1)$, the proof of Step 4 is completed.

\vspace{0.2cm}

To conclude, note that altogether the four steps prove Proposition~\ref{lem box crossing}.
\end{proof}

In the proof of Step 4 above 
we saw how to glue crossings of two parallel copies of the same rectangle to form a longer crossing. In the next remark, we extend this gluing procedure to the setting of two orthogonal copies of the same rectangle, 
as we shall need this later, e.g.\ in the proof of Theorem~\ref{theo:giant} below.
We omit its proof since it requires only minor modifications to the proof of Step $4$ above.

\begin{remark}\label{rem orthogonal}
Fix $\kappa > 1$ and some $\lambda>\lambda_c(p)$. Then, consider the rectangles $\Pi_3 = [0, \kappa R]\times [0,R]$ and $\Pi_4 = [0, R]\times [0, \kappa R]$. Also, fix any positive integer $N\ge 3$ and denote $\alpha_N = 1/2-N^{-1}$ and $\beta_N = 1/2+N^{-1}$. From Proposition~\ref{lem box crossing} we know that the rectangle $\Pi_3' := [1, \kappa R-1]\times [\alpha_N R, \beta_N R]\subseteq \Pi_3$ (respectively $\Pi_4' := [\alpha_N R, \beta_N R]\times [1, \kappa R-1]\subseteq \Pi_4$) is crossed horizontally (vertically, respectively) in $G(\lambda', p)$ a.a.s.\ as $R\to \infty$, where $\lambda'=\tfrac{\lambda+\lambda_c(p)}{2}$. By considering a number of disjoint annuli around the square $[\alpha_N R, \beta_N R]^2$, we deduce as in Step 4 that any horizontal crossing of $\Pi_3'$ and any vertical crossing of $\Pi_4'$ connect within $\Pi_3\cap \Pi_4$ in $G(\lambda, p)$ with probability arbitrarily close to 1 when $N$ is sufficiently large.
\end{remark}

We are now in position to conclude the proof of the fact that a.a.s.\ $G_n$ has a component of linear size.

\begin{proof}[Proof of Theorem~\ref{theo:giant}]
Fix some large $R >0$ so that $\sqrt{n}/R \in \mathbb N$. Then, tessellate the square $\Lambda_n$ into $R\times R$ squares and combine these into horizontal and vertical \emph{dominos}, that is, rectangles of dimensions $2R\times R$ and $R\times 2R$, respectively.
By identifying dominos with vertices and declaring that two dominos are neighbors if they are orthogonal and their intersection is an $R\times R$ square, this defines a graph isomorphic to (a finite subgraph of) $\mathbb Z^2$. 
We now construct an auxiliary random subgraph $H_{\mathrm{aux}}$ by applying the following bond percolation on $\mathbb Z^2$: let $\lambda':= \frac{\lambda+\lambda_c(p)}{2}$, and call a domino \textit{admissible} if a  thinner inner rectangle, as described in Remark~\ref{rem orthogonal}, is crossed by an open path in $G(\lambda',p)$. In this case, we associate to this domino an arbitrary one of these paths. Then, we say that the edge between two neighboring admissible dominos $\Pi_1$ and $\Pi_2$ is open if their corresponding paths are connected by a path in $\Pi_1\cap \Pi_2\cap G(\lambda, p)$. By Proposition~\ref{lem box crossing} and Remark~\ref{rem orthogonal}, every edge is open with a probability that can be made arbitrarily close to $1$ by choosing $R$ sufficiently large (see the proof of Step 4 above). Moreover, the states of two edges at distance larger than $2$ are independent. In other words, $H_{\mathrm{aux}}$ is obtained via a 2--dependent bond percolation on $\mathbb Z^2$. Thus, by applying Theorem~\ref{thm LSS} we deduce that, by choosing $R$ sufficiently large, 
one can ensure that $H_{\mathrm{aux}}$ dominates 
an independent bond percolation with parameter $q = 3/4$, say. Moreover, it is well known that the critical threshold for independent bond percolation on $\mathbb Z^2$ is equal to $1/2$ and also that, in the supercritical regime, by Theorem~\ref{PP theorem}, the largest component in the box $\Lambda_n$ is
a.a.s.\  of size $\Omega_q(n)$ as $n\to \infty$. Finally, notice that, since $R$ is a large but fixed constant, if $H_{\mathrm{aux}}$ contains a component of linear size, then $G_n$ does as well, which concludes the proof of Theorem~\ref{theo:giant}.
\end{proof}

\subsection{The second-largest component}\label{sec.second} 
This section is dedicated to the proof of the following theorem. Recall that we assume throughout that $\lambda$ and $p$ satisfy $\lambda>\lambda_c(p)$. 
\begin{theorem}\label{lem theta log^2}
A.a.s.\ the second-largest component in $G_n(\lambda, p)$ has  size $\Theta_{\lambda,p}((\log n)^2)$.
\end{theorem}

The proof is divided into two parts. We start with the lower bound.

\begin{proposition}\label{lem LB 2nd component}
A.a.s.\ the second-largest component in $G_n(\lambda, p)$ has size $\Omega_{\lambda,p}((\log n)^2)$.
\end{proposition}
\begin{proof}
Tessellate the square $\Lambda_n$ into subsquares of side length $\tfrac{\log n}{5\lambda}$. On the one hand, for any given square of the tessellation, the probability that there is no point of $\mathrm{Po}(\lambda)$ in it at distance at most 1 from its boundary is equal to  $\exp\left(-\lambda\left(\tfrac{4\log n}{5\lambda}-4\right)\right)\ge n^{-4/5}$. On the other hand, by Theorem~\ref{theo:giant}, the probability that in such a square there is a connected component of size $\Omega_{\lambda,p}((\log n)^2)$, whose vertices are all at distance at least one from its boundary, is equal to $1-o(1)$. In particular, for any square in the tessellation, the probability that it contains a
connected component of size $\Omega_{\lambda,p}((\log n)^2)$ which is disconnected from the rest of the graph $G_n$ is at least $(1-o(1)) n^{-4/5}$. Since there are $n^{1-o(1)}$ squares in the tessellation, by Chernoff's inequality a.a.s.\ there is a square that satisfies the above condition, concluding the proof of the proposition.
\end{proof}

We now aim at showing the upper bound. This will be done in several steps. The first point is to show that, roughly speaking, in supercritical Bernoulli bond percolation on $\mathbb Z^2$ large boxes are typically crossed by many disjoint paths (more precisely, the number of crossings typically has the same order as the side length of the box), see Corollary~\ref{cor.Bercross} below for a more precise statement. Then, we use the fact observed in the previous section that, in a sense, the continuous model dominates supercritical Bernoulli bond percolation. We deduce that any small window of side length of order $\log n$ in $\Lambda_n$ contains many disjoint cycles surrounding its center. Finally, some delicate sprinkling allows us to conclude that a.a.s.\ for each vertex $v$ in $\Lambda_n$, no path from $v$ exits the window around $v$ without connecting to at least one of these cycles.

Theorem~\ref{lem theta log^2} will then follow immediately by combining Proposition~\ref{lem LB 2nd component} with the following  
upper bound for the second-largest component, which we state as a separate proposition.
\begin{proposition}\label{prop:second.upper}
A.a.s.\ the second-largest component in $G_n(\lambda, p)$ has size $O_{\lambda,p}((\log n)^2)$.
\end{proposition}

To prove this proposition, we first need some preliminary results on ordinary Bernoulli bond percolation, as mentioned above. 
For $k\ge 1$ and $x\in \mathbb Z^2$, we define $\mathcal C_k(x)$ the {\it $k$-th order connected component of $x$} 
as the set of vertices $y\in \mathbb Z^2$ which are connected to $x$ by a path made of any number of open edges and at most $k-1$ closed edges.  
In particular, the first order connected component is the usual connected component of $x$. Also, for $A\subseteq \mathbb Z^2$, define $\mathcal C_k^A(x)$ 
as the $k$-th order connected component of $x$ in $\mathbb Z^2\setminus A$. We prove the following lemma:  
\begin{lemma}\label{lem:kconnect}
For any $q\in [0,1/2)$, there are positive constants $\alpha$ and $C$ such that, for any positive integers $k$ and $t$, one has 
$$\mathbb P_q\big(\sup_{x\in \mathcal C_k(\bo)} \|x\| \ge t\big)\le C^k \exp(-\alpha t).$$ 
\end{lemma}
\begin{proof}
The proof proceeds by induction on $k$. The result for  $k=1$ follows immediately from Theorem~\ref{DC theorem}. Now assume that it holds for some $k$. Then, for any $t\ge 1$, we aim to obtain an upper bound for $\mathbb P_q(\sup_{x\in \mathcal C_{k+1}(\mathbf{0})}\|x\|\geq t)$. To do this, we take a union bound over all pairs of neighbors $v\in \mathcal C_1(\mathbf{0})$ and $w\in \mathcal C_2(\mathbf{0}) \setminus \mathcal C_1(\mathbf{0})$ (that is, vertex $v$ is in the component of the origin $\mathcal C_1(\mathbf{0})$ and $w$ is a neighbor of $v$ in the complement of $\mathcal C_1(\mathbf{0})$).
Then, by the triangle inequality, the distance of a vertex $x$ to the origin is at most 
$$
\|x\|\leq \|x-w\|+1+\|v\|, 
$$ 
where we use that $\|v-w\|=1$ as $v$ and $w$ are neighbors. As such, using a union bound yields
$$\mathbb P_q\big(\sup_{x\in \mathcal C_{k+1}(\bo)} \|x\| \ge t\big) \le \sum_{\substack{v,w\in \mathbb Z^2 \\  \|w-v\|=1}} \mathbb P_q\big(\bo\leftrightarrow v, \, \bo\centernot\leftrightarrow w, \sup_{x\in \mathcal C_k^{\mathcal C_1(\bo)}(w)} \|x-w\|\ge t -1 - \|v\|\big).$$
Then, by conditioning on $\mathcal C_1(\bo)$ and using the induction hypothesis together with the fact that, for any fixed $A\subseteq \mathbb Z^2$, $\mathcal C_k^A(w)$ is dominated by $\mathcal C_k(w)$, we get 
\begin{align}
\mathbb P_q\big(\sup_{x\in \mathcal C_{k+1}(\bo)} \|x\| \ge t\big) & \le 4\sum_{v\in \mathbb Z^2} \mathbb P_q\big(\bo\leftrightarrow v) \cdot \mathbb P_q\big(\sup_{x\in \mathcal C_k(\bo)} \|x\| \ge t -1 -  \|v\|\big)\nonumber \\
& \le C^k \mathrm{e}^{-\alpha t}\cdot  \Big(4\mathrm{e}^\alpha\sum_{v\in \mathbb Z^2} \mathbb P_q(\bo\leftrightarrow v)\mathrm{e}^{\alpha \|v\|}\Big),\label{eq sum}
\end{align}
where the factor of 4 in the first inequality comes from the fact that every vertex in $\mathbb Z^2$ has 4 neighbors. Applying again Theorem~\ref{DC theorem} yields the existence of a sufficiently small $\alpha$ so that the sum in~\eqref{eq sum} is finite, which concludes the proof of the induction step for $C = 4\mathrm{e}^\alpha\sum_{v\in \mathbb Z^2} \mathbb P_q(\bo\leftrightarrow v)\mathrm{e}^{\alpha \|v\|}$.
\end{proof}

As a corollary, we obtain the following fact.  
\begin{corollary}\label{cor.Bercross}
Fix $q>1/2$. There exist positive constants $C$ and $\alpha$, such that for any positive integers $k$ and $N$, 
$$\mathbb P_q \big(\text{There exists $k$ disjoint horizontal crossings in $[0,2N]\times [0,N]$}\big) \ge 1-  C^k\exp(-\alpha N). $$    
\end{corollary}
\begin{proof}
Consider the dual graph of $\mathbb Z^2$ (its vertices are the faces delimited by the edges of the graph $\mathbb Z^2$, and two faces are adjacent if they share an edge). 
Recall that one can couple Bernoulli bond percolation with parameter $q$ (and $1-q$, respectively) on $\mathbb Z^2$ (and its dual, respectively) by declaring an edge between two vertices of the dual graph open if the edge separating the two corresponding faces of the original graph is closed, and vice versa. Now it suffices to observe that if there are no $k$ disjoint horizontal crossings of the rectangle $[0,2N] \times [0,N]$, then the $k$-th order connected component of one of the vertices, corresponding to the squares with centers $\{(1/2+i, -1/2)\}_{i=0}^{2N-1}$ in the dual graph, reaches the top of the rectangle. By Lemma~\ref{lem:kconnect}, we know that for any such vertex this happens with probability 
at most $C^k \exp(-\alpha N)$ for some positive constants $\alpha$ and $C$. Then, using a union bound over the $2N$ dual vertices on the bottom side and reducing $\alpha$ by factor of 2, say, concludes the proof.
\end{proof}

We are now ready to prove the upper bound on the size of the second-largest component of $G_n$. Its proof goes roughly as follows: To begin with, recall the dominos and the auxiliary graph $H_{\mathrm{aux}}$ from the proof of Theorem~\ref{theo:giant}, and fix a vertex $x\in G_n$.
We perform some delicate exploration of the space and use sprinkling only in a few carefully chosen dominos which have the potential to connect a path from $x$ at distance at most 1 from them with a cycle in $G_n$ going through them. When doing the sprinkling, we require the formation of a dense net of well-connected points after the second stage, thus making sure that all new points participate in the giant.

We will write $G_n^x$ for the percolated random geometric graph with vertex set $(\mathrm{Po}(\lambda) \cup \{x\})\cap \Lambda_n$.  

\begin{proof}[Proof of Proposition~\ref{prop:second.upper}]
Fix $\lambda' = \tfrac{\lambda+\lambda_c(p)}{2}$ and recall the construction of $G(\lambda, p)$ from the independent copies of $G(\lambda', p)$ and $G(\lambda - \lambda', p)$ used in Step 4 of the proof of Proposition~\ref{lem box crossing}. Consider a sufficiently large $R$ so that $\sqrt{n}/R\in \mathbb N$. We construct a version of the auxiliary graph $H_{\textrm{aux}}$ from the proof of Theorem~\ref{theo:giant} as follows: First, reveal $G(\lambda', p)$ and construct a vertex in $H_{\mathrm{aux}}$ if its corresponding domino contains a path as described in Remark~\ref{rem orthogonal}. Then, construct an edge between two neighboring vertices of $H_{\mathrm{aux}}$ if, firstly, the paths in their corresponding dominos $\Pi_1$ and $\Pi_2$ intersect at a point surrounded by at least $K$ disjoint cycles in $G(\lambda',p)$ as in Remark~\ref{rem orthogonal}, with $K$ to be fixed, and secondly, they are connected within $\Pi_1\cap \Pi_2$ in $G(\lambda, p)$. Note that, up to choosing $R$ and $K$ sufficiently large, $H_{\mathrm{aux}}$ dominates a Bernoulli bond percolation with parameter $3/4$ in a box of $\mathbb Z^2$ with side length roughly $\sqrt{n}/R$. We next tessellate this new box into overlapping horizontal and vertical rectangles with side lengths $C_1\log n$ and $2C_1\log n$ called \emph{log-dominos}, where $C_1$ is some large constant satisfying $\frac{\sqrt n}{RC_1\log n}\in \mathbb N$, which will be fixed later. 
Set
$$W:= \left\lfloor \frac{\alpha C_1}{2\log C}\cdot \log n\right\rfloor, $$ 
with $\alpha$ and $C$ as in Corollary~\ref{cor.Bercross}. 
Then, consider the event $\mathcal A_n$ that all log-dominos are crossed by at least $W$ disjoint paths in $H_{\textrm{aux}}$ along their longer side. 
Note that on $\mathcal A_n$, since crossings of two orthogonal overlapping log-dominos have a common vertex of $H_{\mathrm{aux}}$, all these paths are connected in $H_{\textrm{aux}}$ and thus correspond to a connected component $\widehat{G}_n$ of $G_n$ of size $\Theta(n)$.  Moreover, a union bound and Corollary~\ref{cor.Bercross} tell us that if $C_1$ is chosen sufficiently large, then
\begin{equation}\label{probaAn}
\mathbb P(\mathcal A_n) \ge 1- n^2 \exp\left(-\frac{\alpha C_1\log n}{2}\right) = 1- o(1). 
\end{equation}
We aim to bound the probability that there is a connected component of $G_n$ whose diameter is at least $4C_1 \log n$ without being connected to this giant component. 
For this, the strategy will be roughly to show that if such a component exists, say in a square of side length of order $\log n$ centered at a point $X\in \mathrm{Po}(\lambda)\cap \Lambda_n$, 
it will have to cross many cycles surrounding $X$ in this window, and we estimate the cost of this scenario by sprinkling    $G(\lambda-\lambda', p)$ on top of $G(\lambda', p)$. 
Note that some care is required here; indeed, adding a bit of intensity to the Poisson process could also help to create new connected components with large diameter. 
In our proof, we first discover the cycles of  $H_{\mathrm{aux}}$ surrounding $X$ by an exploration from the boundary of a small box around it towards its center, then find  possible paths starting from $X$ and coming close to these cycles, and finally perform a sprinkling only in constant-sized regions with the aim to connect the path from $X$ to any of the cycles with probability bounded away from 0.

\vspace{1em}

To be more precise now, for any $x\in \Lambda_n$, consider the box $\Lambda_n^x:=\Lambda_x(4C_1\log n)$ and define the event 
$$\mathcal O_n^x:=\big\{x\xleftrightarrow{G_n^x\,} \mathbb R^2\setminus \Lambda_n^x\big\},$$ 
that is, the connected component of $x$ in $G_n^x$ reaches the complement of $\Lambda_n^x$. Let also $\mathcal U_n^x$ be the event that in the graph $H_{\textrm{aux}}\cap \Lambda_n^x$ there are $W$ disjoint cycles surrounding $x$ or, in case $x$ is at distance less than $4C_1\log n$ from $\partial \Lambda_n$, paths from the boundary of $H_{\textrm{aux}}$ to itself, which all surround the point $x$ when seen as cycles (parts of which may coincide with the boundary of $\Lambda_n$). 
Note the important fact that, by construction,
\begin{equation}\label{inclusionAnUn}
\mathcal A_n\ \subseteq \ \bigcap_{x\in \Lambda_n} \mathcal U_n^x.
\end{equation} 
On the event $\mathcal U_n^x$, we denote by $\mathcal C_1$ the outermost cycle of $H_{\textrm{aux}}$ in $\Lambda_n^x$ that surrounds $x$ (or possibly the outermost path from the boundary of $H_{\textrm{aux}}$ to itself in case $x$ is at 
distance smaller than $4C_1\log n$ from $\partial \Lambda_n$). 
With a slight abuse of notation, we alternatively view $\mathcal C_1$ as a cycle in the graph 
$H_{\textrm{aux}}$, a set of dominos in $\Lambda_n^x$, or a cycle in $G_n$ obtained by connecting the crossings associated to the sequence of admissible dominos forming the cycle in $H_{\textrm{aux}}$.  
Let also $\mathcal C_1^+$ be the enlarged cycle made of the points in $\Lambda_n^x$ at distance smaller than one from $\mathcal C_1$ (viewed here as a union of dominos).

We next define $\mathcal C_2$ as 
the outermost cycle surrounding $x$ in the restriction of $H_{\textrm{aux}}$ to the region enclosed by $\mathcal C_1^+$. Note that on $\mathcal U_n^x$, by repeating again this procedure, one can define inductively a sequence of disjoint cycles $\mathcal C_3,\dots, \mathcal C_{W/2}$ (from now on, we assume $W$ to be even since decreasing it by 1 does not modify the argument), which all surround $x$ and are at distance at least 1 from each other. Let us define the event $\widetilde{\mathcal U}_n^x$ that the cycles $\mathcal C_1,\dots,\mathcal C_{W/2}$ as constructed above exist. In particular, $\mathcal U_n^x \subseteq \widetilde{\mathcal  U}_n^x$; it will turn out later that it is more convenient to work with the latter for some reasons related to measurability of these events.

Now, for $i=1,\dots,W/2$, denote 
$$\mathcal B_{n,i}^x :=\widetilde{\mathcal U}_n^x\cap  \{\textrm{none of the paths from $x$ to $\mathbb R^2 \setminus \Lambda_n^x$ shares a vertex with $\mathcal  C_i$ in $G_n^x$}\},$$
and 
$$\mathcal B_n^x := \bigcap_{i=1}^{W/2}  \mathcal B_{n,i}^x. $$ 
Recall that on the event $\mathcal A_n$, the cycles $\mathcal C_1,\dots,\mathcal C_{W/2}$, are all part of the same connected component in $G_n$. 
Therefore, if $\mathcal A_n $ holds and the connected component of a point $X\in \mathrm{Po}(\lambda)\cap \Lambda_n$ reaches $\partial \Lambda_n^X$ in $G_n$ without being part of the connected component containing the cycles $\cC_1,\ldots, \cC_{W/2}$, then 
necessarily $\mathcal B_n^X$ must hold. Thus, our aim now is to bound the probability of these events.

We start by observing that a union bound and Lemma~\ref{campbell-mecke} yield 
$$\mathbb P\Big(\bigcup_{X\in \mathrm{Po}(\lambda)\cap \Lambda_n} \mathcal O_n^X\cap \mathcal B_n^X\Big) \le 
\mathbb E\Big[\sum_{X\in  \mathrm{Po}(\lambda)\cap \Lambda_n}  \mathds{1}_{ \mathcal O_n^X\cap\mathcal B_n^X}  \Big] = \lambda \int_{\Lambda_n} \mathbb P\big(\mathcal O_n^x\cap \mathcal B_n^x \big) \, dx, $$
and our goal is thus to bound the probabilities on the right-hand side from above.  

\begin{claim}\label{cl:1}
For any sufficiently large $C_1$ and for every $x\in \Lambda_n$, $\mathbb P(\mathcal O_n^x\cap \mathcal B_n^x) = o(n^{-1})$.
\end{claim}
\begin{proof}[Proof of Claim~\ref{cl:1}]
Fix $x\in \Lambda_n$, and for simplicity assume that it is at distance at least $4C_1\log n$ from $\partial \Lambda_n$; 
the proof in the case of $x$ being closer to $\partial \Lambda_n$ is analogous. 
We first show that the probability that a path in $G_n^x$ between $x$ and $\mathbb R^2\setminus \Lambda_n^x$ does not intersect $\mathcal C_1$ is bounded away from $1$. For this,
note that $\mathcal C_1$ can be determined by exploring the restriction of $H_{\textrm{aux}}$ to $\Lambda_n^x$ by starting from its boundary. More precisely, it suffices to reveal only 
the connected components of the dual graph that touch the 
boundary of $H_{\textrm{aux}}\cap \Lambda_n^x$. 

Let $U_1$ denote the region inside $\Lambda_n^x$ which is enclosed by $\mathcal C_1$ (seen as a set of dominos). Then, conditionally on $\mathcal C_1$, the distribution of the point configuration in $U_1$ is still 
that of an independent Poisson Point Process with intensity $\lambda$. 
Moreover, on the event $\mathcal O_n^x$, the connected component of $x$
in $G_n^x\cap U_1$ contains at least one vertex at distance smaller than one from $\mathcal C_1$. 
We define the \textit{first one} of these vertices as follows. 

First, explore $G(\lambda, p)$ in the entire region $R_1\subseteq U_1$ enclosed by $\cC_1^+$. Then, for $t\ge 0$, let $B_t$ be the set of points at distance at most $t$ from $R_1$, and let 
$$T:=\inf\{t\ge 0 : \text{there exists }y\in \partial B_t \text{ such that }x\xleftrightarrow{G_n^x\cap B_t\,} y\}. $$
Note that $T$ is a stopping time with respect to the filtration $(\mathcal F_t)_{t\ge 0}$ defined by 
$$\mathcal F_t:=\sigma(\mathcal C_1, \, G_n^x \cap B_t), \text{ for all } t\ge 0.$$  
Moreover, as we already mentioned, one has 
\begin{equation}\label{eq inclusion}
\mathcal O_n^x \subseteq \{T<\infty\}. 
\end{equation}

\begin{figure}
\centering
\begin{tikzpicture}[scale=0.82,line cap=round,line join=round,x=1cm,y=1cm]
\clip(-12.1,-7.1) rectangle (14,7.1);
\draw [line width=0.8pt] (-9,5)-- (-9,3);
\draw [line width=0.8pt] (-9,3)-- (-8,3);
\draw [line width=0.8pt] (-8,3)-- (-8,5);
\draw [line width=0.8pt] (-8,5)-- (-9,5);
\draw [line width=0.8pt] (-8,4)-- (-10,4);
\draw [line width=0.8pt] (-10,3)-- (-8,3);
\draw [line width=0.8pt] (-10,4)-- (-10,3);
\draw [line width=0.8pt] (-10,3)-- (-10,2);
\draw [line width=0.8pt] (-11,3)-- (-9,3);
\draw [line width=0.8pt] (-11,3)-- (-11,2);
\draw [line width=0.8pt] (-11,2)-- (-11,1);
\draw [line width=0.8pt] (-11,1)-- (-10,1);
\draw [line width=0.8pt] (-10,1)-- (-10,2);
\draw [line width=0.8pt] (-10,1)-- (-9,1);
\draw [line width=0.8pt] (-10,1)-- (-10,0);
\draw [line width=0.8pt] (-10,0)-- (-9,0);
\draw [line width=0.8pt] (-9,0)-- (-9,1);
\draw [line width=0.8pt] (-8,0)-- (-9,0);
\draw [line width=0.8pt] (-9,0)-- (-9,-1);
\draw [line width=0.8pt] (-9,-1)-- (-8,-1);
\draw [line width=0.8pt] (-8,-1)-- (-8,0);
\draw [line width=0.8pt] (-7,-1)-- (-8,-1);
\draw [line width=0.8pt] (-8,-1)-- (-8,-2);
\draw [line width=0.8pt] (-8,-2)-- (-7,-2);
\draw [line width=0.8pt] (-7,-2)-- (-7,-1);
\draw [line width=0.8pt] (-6,-2)-- (-7,-2);
\draw [line width=0.8pt] (-7,-2)-- (-7,-3);
\draw [line width=0.8pt] (-7,-3)-- (-6,-3);
\draw [line width=0.8pt] (-6,-3)-- (-6,-2);
\draw [line width=0.8pt] (-5,-3)-- (-6,-3);
\draw [line width=0.8pt] (-6,-3)-- (-6,-4);
\draw [line width=0.8pt] (-6,-4)-- (-5,-4);
\draw [line width=0.8pt] (-5,-4)-- (-5,-3);
\draw [line width=0.8pt] (-5,-3)-- (-4,-3);
\draw [line width=0.8pt] (-4,-3)-- (-4,-4);
\draw [line width=0.8pt] (-4,-4)-- (-5,-4);
\draw [line width=0.8pt] (-5,-4)-- (-5,-5);
\draw [line width=0.8pt] (-5,-5)-- (-4,-5);
\draw [line width=0.8pt] (-4,-5)-- (-4,-4);
\draw [line width=0.8pt] (-4,-4)-- (-3,-4);
\draw [line width=0.8pt] (-3,-4)-- (-3,-5);
\draw [line width=0.8pt] (-3,-5)-- (-4,-5);
\draw [line width=0.8pt] (-4,-3)-- (-3,-3);
\draw [line width=0.8pt] (-3,-3)-- (-3,-4);
\draw [line width=0.8pt] (-3,-3)-- (-2,-3);
\draw [line width=0.8pt] (-2,-3)-- (-2,-4);
\draw [line width=0.8pt] (-3,-4)-- (-2,-4);
\draw [line width=0.8pt] (-2,-2)-- (-2,-3);
\draw [line width=0.8pt] (-2,-3)-- (-1,-3);
\draw [line width=0.8pt] (-1,-3)-- (-1,-2);
\draw [line width=0.8pt] (-2,-2)-- (-1,-2);
\draw [line width=0.8pt] (-2,-1)-- (-1,-1);
\draw [line width=0.8pt] (-1,-1)-- (0,-1);
\draw [line width=0.8pt] (0,-2)-- (-1,-2);
\draw [line width=0.8pt] (0,-2)-- (0,-3);
\draw [line width=0.8pt] (0,-3)-- (-1,-3);
\draw [line width=0.8pt] (1,-3)-- (0,-3);
\draw [line width=0.8pt] (0,-3)-- (0,-4);
\draw [line width=0.8pt] (0,-4)-- (1,-4);
\draw [line width=0.8pt] (1,-4)-- (1,-3);
\draw [line width=0.8pt] (2,-4)-- (1,-4);
\draw [line width=0.8pt] (1,-4)-- (1,-5);
\draw [line width=0.8pt] (1,-5)-- (2,-5);
\draw [line width=0.8pt] (2,-5)-- (2,-4);
\draw [line width=0.8pt] (3,-5)-- (2,-5);
\draw [line width=0.8pt] (2,-5)-- (2,-6);
\draw [line width=0.8pt] (2,-6)-- (3,-6);
\draw [line width=0.8pt] (3,-6)-- (3,-5);
\draw [line width=0.8pt] (3,-5)-- (4,-5);
\draw [line width=0.8pt] (4,-5)-- (4,-6);
\draw [line width=0.8pt] (4,-6)-- (3,-6);
\draw [line width=0.8pt] (4,-4)-- (4,-5);
\draw [line width=0.8pt] (4,-4)-- (5,-4);
\draw [line width=0.8pt] (5,-4)-- (5,-5);
\draw [line width=0.8pt] (5,-5)-- (4,-5);
\draw [line width=0.8pt] (5,-3)-- (5,-4);
\draw [line width=0.8pt] (5,-3)-- (6,-3);
\draw [line width=0.8pt] (6,-3)-- (6,-4);
\draw [line width=0.8pt] (6,-4)-- (5,-4);
\draw [line width=0.8pt] (6,-2)-- (6,-3);
\draw [line width=0.8pt] (6,-3)-- (7,-3);
\draw [line width=0.8pt] (7,-3)-- (7,-2);
\draw [line width=0.8pt] (7,-2)-- (6,-2);
\draw [line width=0.8pt] (7,-1)-- (7,-2);
\draw [line width=0.8pt] (7,-1)-- (8,-1);
\draw [line width=0.8pt] (8,-1)-- (8,-2);
\draw [line width=0.8pt] (8,-2)-- (7,-2);
\draw [line width=0.8pt] (7,-1)-- (7,0);
\draw [line width=0.8pt] (7,0)-- (8,0);
\draw [line width=0.8pt] (8,0)-- (8,-1);
\draw [line width=0.8pt] (7,0)-- (6,0);
\draw [line width=0.8pt] (6,0)-- (6,1);
\draw [line width=0.8pt] (6,1)-- (7,1);
\draw [line width=0.8pt] (7,1)-- (7,0);
\draw [line width=0.8pt] (6,1)-- (5,1);
\draw [line width=0.8pt] (5,1)-- (5,2);
\draw [line width=0.8pt] (5,2)-- (6,2);
\draw [line width=0.8pt] (6,2)-- (6,1);
\draw [line width=0.8pt] (5,2)-- (4,2);
\draw [line width=0.8pt] (4,2)-- (4,3);
\draw [line width=0.8pt] (4,3)-- (5,3);
\draw [line width=0.8pt] (5,3)-- (5,2);
\draw [line width=0.8pt] (4,3)-- (3,3);
\draw [line width=0.8pt] (3,3)-- (3,4);
\draw [line width=0.8pt] (3,4)-- (4,4);
\draw [line width=0.8pt] (4,4)-- (4,3);
\draw [line width=0.8pt] (3,4)-- (2,4);
\draw [line width=0.8pt] (2,4)-- (2,5);
\draw [line width=0.8pt] (2,5)-- (3,5);
\draw [line width=0.8pt] (3,5)-- (3,4);
\draw [line width=0.8pt] (2,5)-- (1,5);
\draw [line width=0.8pt] (1,5)-- (1,6);
\draw [line width=0.8pt] (1,6)-- (2,6);
\draw [line width=0.8pt] (2,6)-- (2,5);
\draw [line width=0.8pt] (1,6)-- (0,6);
\draw [line width=0.8pt] (0,6)-- (0,5);
\draw [line width=0.8pt] (0,5)-- (1,5);
\draw [line width=0.8pt] (0,6)-- (0,7);
\draw [line width=0.8pt] (0,7)-- (1,7);
\draw [line width=0.8pt] (1,7)-- (1,6);
\draw [line width=0.8pt] (0,7)-- (-1,7);
\draw [line width=0.8pt] (-1,7)-- (-1,6);
\draw [line width=0.8pt] (-1,6)-- (0,6);
\draw [line width=0.8pt] (-1,6)-- (-1,5);
\draw [line width=0.8pt] (-1,5)-- (0,5);
\draw [line width=0.8pt] (-1,6)-- (-2,6);
\draw [line width=0.8pt] (-2,6)-- (-2,5);
\draw [line width=0.8pt] (-2,5)-- (-1,5);
\draw [line width=0.8pt] (-2,5)-- (-2,4);
\draw [line width=0.8pt] (-2,5)-- (-3,5);
\draw [line width=0.8pt] (-3,5)-- (-3,4);
\draw [line width=0.8pt] (-3,4)-- (-2,4);
\draw [line width=0.8pt] (-3,4)-- (-3,3);
\draw [line width=0.8pt] (-3,4)-- (-4,4);
\draw [line width=0.8pt] (-4,4)-- (-4,3);
\draw [line width=0.8pt] (-4,3)-- (-3,3);
\draw [line width=0.8pt] (-4,3)-- (-4,2);
\draw [line width=0.8pt] (-8,5)-- (-7,5);
\draw [line width=0.8pt] (-7,5)-- (-7,4);
\draw [line width=0.8pt] (-7,4)-- (-8,4);
\draw [line width=0.8pt] (-4,3)-- (-5,3);
\draw [line width=0.8pt] (-5,3)-- (-5,2);
\draw [line width=0.8pt] (-5,2)-- (-4,2);
\draw [line width=0.8pt] (-7,4)-- (-7,3);
\draw [line width=0.8pt] (-7,3)-- (-8,3);
\draw [line width=0.8pt] (-7,4)-- (-6,4);
\draw [line width=0.8pt] (-6,4)-- (-6,3);
\draw [line width=0.8pt] (-6,3)-- (-7,3);
\draw [line width=0.8pt] (-6,2)-- (-6,3);
\draw [line width=0.8pt] (-6,3)-- (-5,3);
\draw [line width=0.8pt] (-6,2)-- (-5,2);
\draw [line width=0.8pt] (-6,1)-- (-5,1);
\draw [line width=0.8pt] (-4,1)-- (-5,1);
\draw [line width=0.8pt] (-1,1)-- (-0.5227866014611782,0.8265746986137462);
\draw [line width=0.8pt] (-0.5227866014611782,0.8265746986137462)-- (-0.5676971253078594,1.3205904609272374);
\draw [line width=0.8pt] (-0.5676971253078594,1.3205904609272374)-- (0,1);
\draw [line width=0.8pt] (0,1)-- (-0.268293632996651,0.5571115555336601);
\draw [line width=0.8pt] (-0.268293632996651,0.5571115555336601)-- (0.3604537008568866,0.4972308570714188);
\draw [line width=0.8pt] (0.3604537008568866,0.4972308570714188)-- (0.6448870185525346,1.0211869686160306);
\draw [line width=0.8pt] (0.6448870185525346,1.0211869686160306)-- (0.46524492316580957,1.5601132547762029);
\draw [line width=0.8pt] (0.46524492316580957,1.5601132547762029)-- (1.3484852254838744,1.8295763978562891);
\draw [line width=0.8pt] (1.3484852254838744,1.8295763978562891)-- (1,2);
\draw [line width=0.8pt] (1,2)-- (0.9001177148635403,2.2652819784377423);
\draw [line width=0.8pt] (0.9001177148635403,2.2652819784377423)-- (1.2239941791889504,2.5729646195468736);
\draw [line width=0.8pt] (1.2239941791889504,2.5729646195468736)-- (1.4237179988562867,2.1519252159238516);
\draw [line width=0.8pt] (1.4237179988562867,2.1519252159238516)-- (1.9041347542723117,1.9468034551844307);
\draw [line width=0.8pt] (-6,4)-- (-5,4);
\draw [line width=0.8pt] (-5,4)-- (-5,3);
\draw [line width=0.8pt] (-5,4)-- (-4,4);
\draw [line width=0.8pt,color=blue] (-5.506210441055699,2.5043443878652885)-- (-4.5,2.5);
\draw [line width=0.8pt,color=blue] (-4.5,2.5)-- (-4.5,3.5);
\draw [line width=0.8pt,color=blue] (-4.5,3.5)-- (-3.5,3.5);
\draw [line width=0.8pt,color=blue] (-3.5,3.5)-- (-3.5,4.5);
\draw [line width=0.8pt,color=blue] (-3.5,4.5)-- (-2.5,4.5);
\draw [shift={(-6,2)},line width=0.8pt]  plot[domain=3.141592653589793:4.71238898038469,variable=\t]({1*1*cos(\t r)+0*1*sin(\t r)},{0*1*cos(\t r)+1*1*sin(\t r)});
\draw [line width=0.8pt,color=blue] (-5.506210441055699,2.5043443878652885)-- (-5.5,3.5);
\draw [line width=0.8pt,color=blue] (-5.5,3.5)-- (-6.5,3.5);
\draw [line width=0.8pt,color=blue] (-6.5,3.5)-- (-6.5,4.5);
\draw [line width=0.8pt,color=blue] (-6.5,4.5)-- (-7.5,4.5);
\draw [line width=0.8pt,color=blue] (-7.5,4.5)-- (-7.5,3.5);
\draw [line width=0.8pt,color=blue] (-7.5,3.5)-- (-8.5,3.5);
\draw [line width=0.8pt,color=blue] (-8.5,3.5)-- (-8.5,4.5);
\draw [line width=0.8pt,color=blue] (-8.5,4.5)-- (-9.5,4.5);
\draw [line width=0.8pt,color=blue] (-9.5,4.5)-- (-9.5,3.5);
\draw [line width=0.8pt,color=blue] (-9.5,3.5)-- (-10.5,3.5);
\draw [line width=0.8pt,color=blue] (-10.5,3.5)-- (-10.5,2.5);
\draw [line width=0.8pt,color=blue] (-10.5,2.5)-- (-11.5,2.5);
\draw [line width=0.8pt,color=blue] (-11.5,2.5)-- (-11.5,1.5);
\draw [line width=0.8pt,color=blue] (-11.5,1.5)-- (-10.5,1.5);
\draw [line width=0.8pt,color=blue] (-10.5,1.5)-- (-10.5,0.5);
\draw [line width=0.8pt,color=blue] (-10.5,0.5)-- (-9.5,0.5);
\draw [line width=0.8pt,color=blue] (-9.5,0.5)-- (-9.5,-0.5);
\draw [line width=0.8pt,color=blue] (-9.5,-0.5)-- (-8.5,-0.5);
\draw [line width=0.8pt,color=blue] (-8.5,-0.5)-- (-8.5,-1.5);
\draw [line width=0.8pt,color=blue] (-8.5,-1.5)-- (-7.5,-1.5);
\draw [line width=0.8pt,color=blue] (-7.5,-1.5)-- (-7.5,-2.5);
\draw [line width=0.8pt,color=blue] (-7.5,-2.5)-- (-6.5,-2.5);
\draw [line width=0.8pt,color=blue] (-6.5,-2.5)-- (-6.5,-3.5);
\draw [line width=0.8pt,color=blue] (-6.5,-3.5)-- (-5.5,-3.5);
\draw [line width=0.8pt,color=blue] (-5.5,-3.5)-- (-5.5,-4.5);
\draw [line width=0.8pt,color=blue] (-5.5,-4.5)-- (-4.5,-4.5);
\draw [line width=0.8pt,color=blue] (-4.5,-4.5)-- (-4.5,-3.5);
\draw [line width=0.8pt,color=blue] (-4.5,-3.5)-- (-3.5,-3.5);
\draw [line width=0.8pt,color=blue] (-3.5,-3.5)-- (-3.5,-4.5);
\draw [line width=0.8pt,color=blue] (-3.5,-4.5)-- (-2.5,-4.5);
\draw [line width=0.8pt,color=blue] (-2.5,-4.5)-- (-2.5,-3.5);
\draw [line width=0.8pt,color=blue] (-2.5,-3.5)-- (-1.5,-3.5);
\draw [line width=0.8pt,color=blue] (-1.5,-3.5)-- (-1.5,-2.5);
\draw [line width=0.8pt,color=blue] (-1.5,-2.5)-- (-0.5,-2.5);
\draw [line width=0.8pt,color=blue] (-0.5,-2.5)-- (-0.5,-3.5);
\draw [line width=0.8pt,color=blue] (-0.5,-3.5)-- (0.5,-3.5);
\draw [line width=0.8pt,color=blue] (0.5,-3.5)-- (0.5,-4.5);
\draw [line width=0.8pt,color=blue] (0.5,-4.5)-- (1.5,-4.5);
\draw [line width=0.8pt,color=blue] (1.5,-4.5)-- (1.5,-5.5);
\draw [line width=0.8pt,color=blue] (1.5,-5.5)-- (2.5,-5.5);
\draw [line width=0.8pt,color=blue] (2.5,-5.5)-- (2.5,-6.5);
\draw [line width=0.8pt,color=blue] (2.5,-6.5)-- (3.5,-6.5);
\draw [line width=0.8pt,color=blue] (3.5,-6.5)-- (3.5,-5.5);
\draw [line width=0.8pt,color=blue] (3.5,-5.5)-- (4.5,-5.5);
\draw [line width=0.8pt,color=blue] (4.5,-5.5)-- (4.5,-4.5);
\draw [line width=0.8pt,color=blue] (4.5,-4.5)-- (5.5,-4.5);
\draw [line width=0.8pt,color=blue] (5.5,-4.5)-- (5.5,-3.5);
\draw [line width=0.8pt,color=blue] (5.5,-3.5)-- (6.5,-3.5);
\draw [line width=0.8pt,color=blue] (6.5,-3.5)-- (6.5,-2.5);
\draw [line width=0.8pt,color=blue] (6.5,-2.5)-- (7.5,-2.5);
\draw [line width=0.8pt,color=blue] (7.5,-2.5)-- (7.5,-1.5);
\draw [line width=0.8pt,color=blue] (7.5,-1.5)-- (8.5,-1.5);
\draw [line width=0.8pt,color=blue] (8.5,-1.5)-- (8.5,-0.5);
\draw [line width=0.8pt,color=blue] (8.5,-0.5)-- (7.5,-0.5);
\draw [line width=0.8pt,color=blue] (7.5,-0.5)-- (7.5,0.5);
\draw [line width=0.8pt,color=blue] (7.5,0.5)-- (6.5,0.5);
\draw [line width=0.8pt,color=blue] (6.5,0.5)-- (6.5,1.5);
\draw [line width=0.8pt,color=blue] (6.5,1.5)-- (5.5,1.5);
\draw [line width=0.8pt,color=blue] (5.5,1.5)-- (5.5,2.5);
\draw [line width=0.8pt,color=blue] (5.5,2.5)-- (4.5,2.5);
\draw [line width=0.8pt,color=blue] (4.5,2.5)-- (4.5,3.5);
\draw [line width=0.8pt,color=blue] (4.5,3.5)-- (3.5,3.5);
\draw [line width=0.8pt,color=blue] (3.5,3.5)-- (3.5,4.5);
\draw [line width=0.8pt,color=blue] (3.5,4.5)-- (2.5,4.5);
\draw [line width=0.8pt,color=blue] (2.5,4.5)-- (2.5,5.5);
\draw [line width=0.8pt,color=blue] (2.5,5.5)-- (1.5,5.5);
\draw [line width=0.8pt,color=blue] (1.5,5.5)-- (1.5,6.5);
\draw [line width=0.8pt,color=blue] (1.5,6.5)-- (0.5,6.5);
\draw [line width=0.8pt,color=blue] (0.5,6.5)-- (0.5,5.5);
\draw [line width=0.8pt,color=blue] (0.5,5.5)-- (-0.5,5.5);
\draw [line width=0.8pt,color=blue] (-0.5,5.5)-- (-0.5,6.5);
\draw [line width=0.8pt,color=blue] (-0.5,6.5)-- (-1.5,6.5);
\draw [line width=0.8pt,color=blue] (-1.5,6.5)-- (-1.5,5.5);
\draw [line width=0.8pt,color=blue] (-1.5,5.5)-- (-2.5,5.5);
\draw [line width=0.8pt,color=blue] (-2.5,5.5)-- (-2.5,4.5);
\draw [line width=0.8pt] (1,7)-- (2,7);
\draw [line width=0.8pt] (2,7)-- (2,6);
\draw [line width=0.8pt] (2,6)-- (3,6);
\draw [line width=0.8pt] (3,6)-- (3,5);
\draw [line width=0.8pt] (3,5)-- (4,5);
\draw [line width=0.8pt] (4,5)-- (4,4);
\draw [line width=0.8pt] (4,4)-- (5,4);
\draw [line width=0.8pt] (5,4)-- (5,3);
\draw [line width=0.8pt] (5,3)-- (6,3);
\draw [line width=0.8pt] (6,3)-- (6,2);
\draw [line width=0.8pt] (6,2)-- (7,2);
\draw [line width=0.8pt] (7,2)-- (7,1);
\draw [line width=0.8pt] (7,1)-- (8,1);
\draw [line width=0.8pt] (-1,7)-- (-2,7);
\draw [line width=0.8pt] (-2,7)-- (-2,6);
\draw [line width=0.8pt] (-2,6)-- (-3,6);
\draw [line width=0.8pt] (-3,6)-- (-3,5);
\draw [line width=0.8pt] (-3,5)-- (-4,5);
\draw [line width=0.8pt] (-4,5)-- (-4,4);
\draw [line width=0.8pt] (8,1)-- (8,0);
\draw [line width=0.8pt] (8,0)-- (9,0);
\draw [line width=0.8pt] (9,0)-- (9,-1);
\draw [line width=0.8pt] (9,-1)-- (8,-1);
\draw [line width=0.8pt] (9,-1)-- (9,-2);
\draw [line width=0.8pt] (9,-2)-- (8,-2);
\draw [line width=0.8pt] (8,-2)-- (8,-3);
\draw [line width=0.8pt] (8,-3)-- (7,-3);
\draw [line width=0.8pt] (7,-3)-- (7,-4);
\draw [line width=0.8pt] (7,-4)-- (6,-4);
\draw [line width=0.8pt] (6,-4)-- (6,-5);
\draw [line width=0.8pt] (6,-5)-- (5,-5);
\draw [line width=0.8pt] (5,-5)-- (5,-6);
\draw [line width=0.8pt] (5,-6)-- (4,-6);
\draw [line width=0.8pt] (-6,4)-- (-6,5);
\draw [line width=0.8pt] (-6,5)-- (-7,5);
\draw [line width=0.8pt] (-9,5)-- (-10,5);
\draw [line width=0.8pt] (-10,5)-- (-10,4);
\draw [line width=0.8pt] (-10,4)-- (-11,4);
\draw [line width=0.8pt] (-11,4)-- (-11,3);
\draw [line width=0.8pt] (-11,3)-- (-12,3);
\draw [line width=0.8pt] (-12,3)-- (-12,2);
\draw [line width=0.8pt] (-12,2)-- (-12,1);
\draw [line width=0.8pt] (-12,1)-- (-11,1);
\draw [line width=0.8pt] (-12,2)-- (-11,2);
\draw [line width=0.8pt] (-6,-4)-- (-7,-4);
\draw [line width=0.8pt] (-7,-4)-- (-7,-3);
\draw [line width=0.8pt] (-7,-3)-- (-8,-3);
\draw [line width=0.8pt] (-8,-3)-- (-8,-2);
\draw [line width=0.8pt] (-8,-2)-- (-9,-2);
\draw [line width=0.8pt] (-9,-2)-- (-9,-1);
\draw [line width=0.8pt] (-9,-1)-- (-10,-1);
\draw [line width=0.8pt] (-10,-1)-- (-10,0);
\draw [line width=0.8pt] (-11,1)-- (-11,0);
\draw [line width=0.8pt] (-11,0)-- (-10,0);
\draw [line width=0.8pt] (-6,-4)-- (-6,-5);
\draw [line width=0.8pt] (-6,-5)-- (-5,-5);
\draw [line width=0.8pt] (2,-6)-- (2,-7);
\draw [line width=0.8pt] (2,-7)-- (3,-7);
\draw [line width=0.8pt] (3,-7)-- (3,-6);
\draw [line width=0.8pt] (4,-6)-- (4,-7);
\draw [line width=0.8pt] (4,-7)-- (3,-7);
\draw [shift={(-2,4)},line width=0.8pt]  plot[domain=-1.5707963267948966:0,variable=\t]({1*1*cos(\t r)+0*1*sin(\t r)},{0*1*cos(\t r)+1*1*sin(\t r)});
\draw [shift={(2,4)},line width=0.8pt]  plot[domain=3.141592653589793:4.71238898038469,variable=\t]({1*1*cos(\t r)+0*1*sin(\t r)},{0*1*cos(\t r)+1*1*sin(\t r)});
\draw [line width=0.8pt] (-1,4)-- (1,4);
\draw [shift={(-3,3)},line width=0.8pt]  plot[domain=-1.5707963267948966:0,variable=\t]({1*1*cos(\t r)+0*1*sin(\t r)},{0*1*cos(\t r)+1*1*sin(\t r)});
\draw [shift={(-4,2)},line width=0.8pt]  plot[domain=-1.5707963267948966:0,variable=\t]({1*1*cos(\t r)+0*1*sin(\t r)},{0*1*cos(\t r)+1*1*sin(\t r)});
\draw [line width=0.8pt] (-11,2)-- (-10,2);
\draw [line width=0.8pt] (-7,2)-- (-9,2);
\draw [shift={(-9,1)},line width=0.8pt]  plot[domain=0:1.5707963267948966,variable=\t]({1*1*cos(\t r)+0*1*sin(\t r)},{0*1*cos(\t r)+1*1*sin(\t r)});
\draw [shift={(-8,0)},line width=0.8pt]  plot[domain=0:1.5787138138642733,variable=\t]({1*1*cos(\t r)+0*1*sin(\t r)},{0*1*cos(\t r)+1*1*sin(\t r)});
\draw [shift={(-7,-1)},line width=0.8pt]  plot[domain=0:1.5707963267948966,variable=\t]({1*1*cos(\t r)+0*1*sin(\t r)},{0*1*cos(\t r)+1*1*sin(\t r)});
\draw [shift={(-6,-2)},line width=0.8pt]  plot[domain=0:1.5707963267948966,variable=\t]({1*1*cos(\t r)+0*1*sin(\t r)},{0*1*cos(\t r)+1*1*sin(\t r)});
\draw [line width=0.8pt] (-5,-2)-- (-4,-2);
\draw [line width=0.8pt] (-4,-2)-- (-3,-2);
\draw [shift={(-2,-2)},line width=0.8pt]  plot[domain=1.5707963267948966:3.141592653589793,variable=\t]({1*1*cos(\t r)+0*1*sin(\t r)},{0*1*cos(\t r)+1*1*sin(\t r)});
\draw [shift={(0,-2)},line width=0.8pt]  plot[domain=0:1.5707963267948966,variable=\t]({1*1*cos(\t r)+0*1*sin(\t r)},{0*1*cos(\t r)+1*1*sin(\t r)});
\draw [shift={(1,-3)},line width=0.8pt]  plot[domain=0:1.5707963267948966,variable=\t]({1*1*cos(\t r)+0*1*sin(\t r)},{0*1*cos(\t r)+1*1*sin(\t r)});
\draw [shift={(2,-4)},line width=0.8pt]  plot[domain=0:1.5707963267948966,variable=\t]({1*1*cos(\t r)+0*1*sin(\t r)},{0*1*cos(\t r)+1*1*sin(\t r)});
\draw [line width=0.8pt] (0,-4)-- (0,-5);
\draw [line width=0.8pt] (0,-5)-- (1,-5);
\draw [shift={(4,-4)},line width=0.8pt]  plot[domain=1.5707963267948966:3.141592653589793,variable=\t]({1*1*cos(\t r)+0*1*sin(\t r)},{0*1*cos(\t r)+1*1*sin(\t r)});
\draw [shift={(5,-3)},line width=0.8pt]  plot[domain=1.5707963267948966:3.141592653589793,variable=\t]({1*1*cos(\t r)+0*1*sin(\t r)},{0*1*cos(\t r)+1*1*sin(\t r)});
\draw [shift={(6,-2)},line width=0.8pt]  plot[domain=1.5707963267948966:3.141592653589793,variable=\t]({1*1*cos(\t r)+0*1*sin(\t r)},{0*1*cos(\t r)+1*1*sin(\t r)});
\draw [shift={(6,0)},line width=0.8pt]  plot[domain=3.141592653589793:4.71238898038469,variable=\t]({1*1*cos(\t r)+0*1*sin(\t r)},{0*1*cos(\t r)+1*1*sin(\t r)});
\draw [shift={(5,1)},line width=0.8pt]  plot[domain=3.141592653589793:4.71238898038469,variable=\t]({1*1*cos(\t r)+0*1*sin(\t r)},{0*1*cos(\t r)+1*1*sin(\t r)});
\draw [shift={(4,2)},line width=0.8pt]  plot[domain=3.141592653589793:4.71238898038469,variable=\t]({1*1*cos(\t r)+0*1*sin(\t r)},{0*1*cos(\t r)+1*1*sin(\t r)});
\draw [shift={(3,3)},line width=0.8pt]  plot[domain=3.141592653589793:4.71238898038469,variable=\t]({1*1*cos(\t r)+0*1*sin(\t r)},{0*1*cos(\t r)+1*1*sin(\t r)});
\draw [line width=0.8pt,color=red] (-2,-0.8)-- (0,-0.8);
\draw [line width=0.8pt,color=red] (0.8167840433800762,3.8)-- (-0.8167840433800762,3.8);
\draw [shift={(-2,4)},line width=0.8pt,color=red]  plot[domain=4.867649897717572:6.1157372279598965,variable=\t]({1*1.2*cos(\t r)+0*1.2*sin(\t r)},{0*1.2*cos(\t r)+1*1.2*sin(\t r)});
\draw [shift={(-3,3)},line width=0.8pt,color=red]  plot[domain=4.867649897717573:6.127924389846702,variable=\t]({1*1.2*cos(\t r)+0*1.2*sin(\t r)},{0*1.2*cos(\t r)+1*1.2*sin(\t r)});
\draw [shift={(-4,2)},line width=0.8pt,color=red]  plot[domain=4.71238898038469:6.1279243898467035,variable=\t]({1*1.2*cos(\t r)+0*1.2*sin(\t r)},{0*1.2*cos(\t r)+1*1.2*sin(\t r)});
\draw [shift={(-6,2)},line width=0.8pt,color=red]  plot[domain=3.309040732809486:4.71238898038469,variable=\t]({1*1.2*cos(\t r)+0*1.2*sin(\t r)},{0*1.2*cos(\t r)+1*1.2*sin(\t r)});
\draw [line width=0.8pt,color=red] (-6,0.8)-- (-4,0.8);
\draw [line width=0.8pt,color=red] (-7.183215956619923,1.8)-- (-8.105572809000078,1.8);
\draw [shift={(-9,1)},line width=0.8pt,color=red]  plot[domain=0.1552609173328838:0.7297276562269601,variable=\t]({1*1.2*cos(\t r)+0*1.2*sin(\t r)},{0*1.2*cos(\t r)+1*1.2*sin(\t r)});
\draw [shift={(-8,0)},line width=0.8pt,color=red]  plot[domain=0.1552609173328838:1.415535409462013,variable=\t]({1*1.2*cos(\t r)+0*1.2*sin(\t r)},{0*1.2*cos(\t r)+1*1.2*sin(\t r)});
\draw [shift={(-7,-1)},line width=0.8pt,color=red]  plot[domain=0.15526091733288314:1.415535409462013,variable=\t]({1*1.2*cos(\t r)+0*1.2*sin(\t r)},{0*1.2*cos(\t r)+1*1.2*sin(\t r)});
\draw [shift={(-6,-2)},line width=0.8pt,color=red]  plot[domain=0.1674480792196893:1.4155354094620134,variable=\t]({1*1.2*cos(\t r)+0*1.2*sin(\t r)},{0*1.2*cos(\t r)+1*1.2*sin(\t r)});
\draw [line width=0.8pt,color=red] (-4.816784043380077,-1.8)-- (-3.183215956619923,-1.8);
\draw [shift={(-2,-2)},line width=0.8pt,color=red]  plot[domain=1.5707963267948966:2.974144574370104,variable=\t]({1*1.2*cos(\t r)+0*1.2*sin(\t r)},{0*1.2*cos(\t r)+1*1.2*sin(\t r)});
\draw [shift={(0,-2)},line width=0.8pt,color=red]  plot[domain=0.1552609173328845:1.5707963267948966,variable=\t]({1*1.2*cos(\t r)+0*1.2*sin(\t r)},{0*1.2*cos(\t r)+1*1.2*sin(\t r)});
\draw [shift={(1,-3)},line width=0.8pt,color=red]  plot[domain=0.15526091733288439:1.415535409462013,variable=\t]({1*1.2*cos(\t r)+0*1.2*sin(\t r)},{0*1.2*cos(\t r)+1*1.2*sin(\t r)});
\draw [shift={(2,-4)},line width=0.8pt,color=red]  plot[domain=0.585685543457152:1.4155354094620123,variable=\t]({1*1.2*cos(\t r)+0*1.2*sin(\t r)},{0*1.2*cos(\t r)+1*1.2*sin(\t r)});
\draw [shift={(4,-4)},line width=0.8pt,color=red]  plot[domain=1.7260572441277793:2.555907110132641,variable=\t]({1*1.2*cos(\t r)+0*1.2*sin(\t r)},{0*1.2*cos(\t r)+1*1.2*sin(\t r)});
\draw [shift={(5,-3)},line width=0.8pt,color=red]  plot[domain=1.72605724412778:2.986331736256909,variable=\t]({1*1.2*cos(\t r)+0*1.2*sin(\t r)},{0*1.2*cos(\t r)+1*1.2*sin(\t r)});
\draw [shift={(6,-2)},line width=0.8pt,color=red]  plot[domain=2.156481870252046:2.98633173625691,variable=\t]({1*1.2*cos(\t r)+0*1.2*sin(\t r)},{0*1.2*cos(\t r)+1*1.2*sin(\t r)});
\draw [shift={(4,2)},line width=0.8pt,color=red]  plot[domain=3.2968535709226763:4.557128063051806,variable=\t]({1*1.2*cos(\t r)+0*1.2*sin(\t r)},{0*1.2*cos(\t r)+1*1.2*sin(\t r)});
\draw [shift={(5,1)},line width=0.8pt,color=red]  plot[domain=3.2968535709226763:4.557128063051805,variable=\t]({1*1.2*cos(\t r)+0*1.2*sin(\t r)},{0*1.2*cos(\t r)+1*1.2*sin(\t r)});
\draw [shift={(6,0)},line width=0.8pt,color=red]  plot[domain=3.296853570922679:4.12670343692754,variable=\t]({1*1.2*cos(\t r)+0*1.2*sin(\t r)},{0*1.2*cos(\t r)+1*1.2*sin(\t r)});
\draw [shift={(2,4)},line width=0.8pt,color=red]  plot[domain=3.3090407328094824:4.557128063051806,variable=\t]({1*1.2*cos(\t r)+0*1.2*sin(\t r)},{0*1.2*cos(\t r)+1*1.2*sin(\t r)});
\draw [shift={(3,3)},line width=0.8pt,color=red]  plot[domain=3.2968535709226776:4.557128063051804,variable=\t]({1*1.2*cos(\t r)+0*1.2*sin(\t r)},{0*1.2*cos(\t r)+1*1.2*sin(\t r)});
\draw [line width=0.8pt] (1.9041347542723117,1.9468034551844307)-- (1.9952219584748,2.3439351501042127);
\draw [line width=0.8pt] (-3,-5)-- (-2,-5);
\draw [line width=0.8pt] (-2,-5)-- (-2,-4);
\draw [line width=0.8pt] (-2,-4)-- (-1,-4);
\draw [line width=0.8pt] (-1,-4)-- (0,-4);
\draw [line width=0.8pt] (-1,-3)-- (-1,-4);
\draw [line width=0.8pt] (1,-5)-- (1,-6);
\draw [line width=0.8pt] (1,-6)-- (2,-6);
\begin{scriptsize}
\draw[color=black] (3.35,3.3) node {$Q_T$};

\draw [fill=black] (-1,1) circle (1pt);
\draw[color=black] (-1.2,1.1) node {$x$};

\draw [fill=black] (-5.500892181475667,0.9146077829851244) circle (1pt);
\draw [fill=black] (-2.769991007314686,-1.201469204755225) circle (1pt);
\draw [fill=black] (4,0.5) circle (1pt);

\draw [fill=black] (1.9952219584748,2.3439351501042127) circle (1pt);
\draw[color=black] (1.65,2.4) node {$X_T$};
\end{scriptsize}
\end{tikzpicture}
\caption{The figure depicts $x$, $X_T$ and the square $Q_T$. The dominos in the cycle $\mathcal{C}_1$ are the ones intersecting the blue curve. The black contour corresponds to the boundary of $\mathcal{C}_1^+$ while the red contour is the boundary of $B_T$.}
\label{fig:1'}
\end{figure}
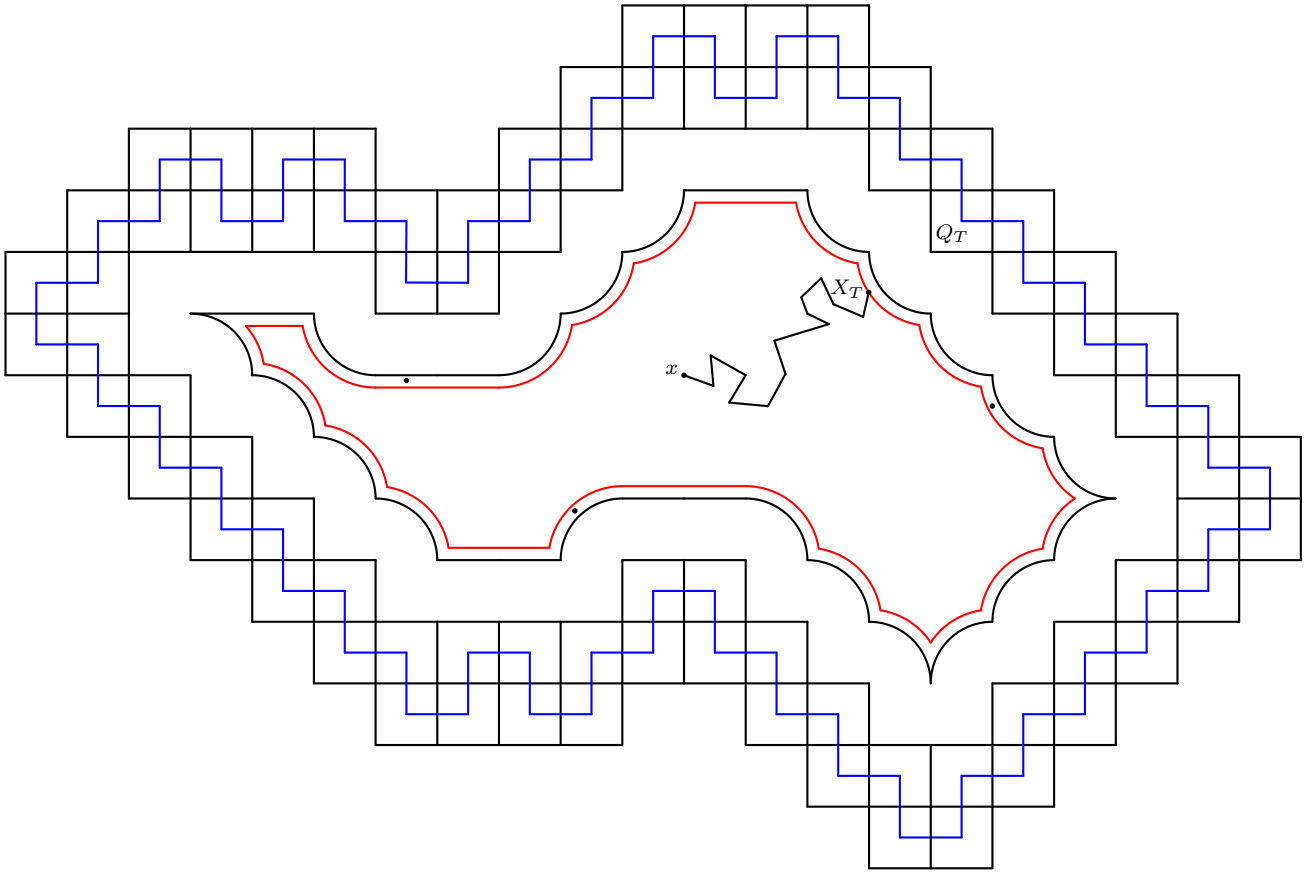

Let $X_T$ be the (almost surely unique) point on $\partial B_T$ which is connected to $x$ in $G_n^x\cap B_T$. Let $Q_T$ be the (almost surely unique) closest $R\times R$ square from $X_T$  which is part of  
$\mathcal C_1$ (when viewing $\mathcal C_1$ as a succession of $R\times R$ squares by dividing every domino in half), and let $S_T : =\Lambda_{X_T}(1) \cap (U_1\setminus B_T)$ be the part of the box $\Lambda_{X_T}(1)$ which is still unexplored when only $\mathcal C_1$ and $B_T$ has
been revealed, see Figure~\ref{fig:1'}. We now define an event $\mathcal E_n^x$ on which $X_T$ connects to $\mathcal C_1$. Firstly, recall that we set $\lambda' = \frac{\lambda + \lambda_c(p)}{2}$, 
and we view $\mathrm{Po}(\lambda)$ as the union of two independent Poisson Point Processes with intensities $\lambda'$ and $\lambda- \lambda'$, respectively. Then, consider a tessellation into squares of side length $1/\sqrt 5$ of $Q_T$, and let $\mathcal E_n^x=\bigcap_{i\leq 3}\mathcal E_{n,i}^x$ be the event that the following 
three conditions are satisfied: 
\begin{itemize}
\item $\mathcal E_{n,1}^x$: There exists a set $\mathcal P\subseteq \mathrm{Po}(\lambda-\lambda')$ containing exactly one point in every square of the tessellation, and such that every pair of points of $\mathcal P$ in adjacent squares is connected by an edge in $G(\lambda-\lambda', p)$ (points in adjacent squares are always at distance smaller than one by construction). 
\item $\mathcal E_{n,2}^x$: One of the points in $\mathcal P$ is connected by an edge between $G(\lambda-\lambda', p)$ and $G(\lambda', p)$ to the path associated to one of the dominos in $\mathcal C_1$ containing $Q_T$ (we recall that by definition these paths are part of $G(\lambda',p)$). 
\item $\mathcal E_{n,3}^x$: If $X_T$ is at distance smaller than $0.1$ from $Q_T$, we ask that it is connected by an edge in $G(\lambda-\lambda', p)$ to the closest point of $\mathcal P$ (this is possible since $0.1 + \sqrt{2/5} < 1$, where $\sqrt{2/5}$ is the length of a diagonal of a square in the tessellation).
If not, we require that in $S_T$ there is a point of $\mathrm{Po}(\lambda-\lambda')$ that is connected by an edge in $G(\lambda-\lambda',p)$ to both $X_T$ and a  point of $\mathcal P$.
\end{itemize}
Note that, conditionally on $G(\lambda',p)$, $\mathcal E_n^x$ is an increasing event which is measurable with respect to $G(\lambda-\lambda',p)$. Moreover, by construction, conditionally on the event $\mathcal E_n^x$, $X_T$ is connected to $\mathcal C_1$ (seen here as a cycle in the original graph $G_n$) in $G(\lambda,p)$. We claim that there exists $\varepsilon>0$ (only depending on $R$, $\lambda$ and $p$) such that almost surely
\begin{equation}\label{claim.Enx}
\mathbb P( \mathcal E_n^x \mid \mathcal F_T)\cdot \mathds{1}_{\widetilde{\mathcal U}_n^x\cap \{T<\infty\}} \ge \varepsilon \cdot \mathds{1}_{\widetilde{\mathcal U}_n^x\cap \{T<\infty\}}.
\end{equation}
(Note that, while $T$ is clearly $\cF_T$-measurable, $\widetilde{\mathcal U}_n^x$ is also $\mathcal F_T$-measurable since each of $\cC_2, \ldots, \cC_{W/2}$ is included in $R_1$. 
In particular, this explains why it is more convenient to work with $\widetilde{\mathcal U}_n^x$ instead of $\mathcal U_n^x$.) 
To see this, notice first that, since $T$ is a stopping time, conditionally on $\mathcal F_T$, the distribution of $G_n^x$ in the region $U_1\setminus B_T$ is that of an independent percolated random geometric graph. 
Thus, given that the first two conditions are satisfied, the third one is achieved at a constant cost independently of $\mathcal F_T$. Otherwise said, there is $\eps' > 0$, such that
\[\mathbb P(\mathcal E_n^x \mid \mathcal F_T)\cdot \mathds{1}_{\widetilde{\mathcal U}_n^x\cap \{T<\infty\}} \ge \eps'\cdot \mathbb P(\mathcal E_{n,1}^x\cap \mathcal E_{n,2}^x \mid \mathcal F_T)\cdot \mathds{1}_{\widetilde{\mathcal U}_n^x\cap \{T<\infty\}}.\]
On the other hand, the first two conditions are more delicate to handle.
Observe that, by construction, conditionally on $\mathcal F_T$, the only information that we have on $Q_T$ is 
that it is part of one, two or three admissible dominos whose associated paths get connected in $Q_T$ when sprinkling $\mathrm{Po}(\lambda-\lambda')$. 
In particular, all dominos disjoint from the interior of $U_1$ that contain $Q_T$ are admissible and participate in the cycle $\cC_1$. 
This has as an important consequence that, conditionally on $G(\lambda,p)$ in $\Lambda_n^x\setminus (Q_T\cup U_1)$ and on $G(\lambda',p)$ in $Q_T$, the event that $Q_T$ indeed closes the cycle $\mathcal C_1$ when sprinkling $\mathrm{Po}(\lambda - \lambda')$ is increasing. 
Therefore, by the FKG inequality (see Lemma~\ref{lem harris}), conditionally on this event, the probability that the first two conditions of $\mathcal E_n^x$ are satisfied is greater than its corresponding probability without the conditioning. 
More precisely,
\begin{align*}
\mathbb P(\mathcal E_{n,1}^x\cap \mathcal E_{n,2}^x \mid \mathcal F_T)\cdot \mathds{1}_{\widetilde{\mathcal U}_n^x\cap \{T<\infty\}}
&= \mathbb P(\mathcal E_{n,1}^x\mid \mathcal F_T)\cdot \mathbb P(\mathcal E_{n,2}^x\mid \mathcal E_{n,1}^x, \mathcal F_T)\cdot \mathds{1}_{\widetilde{\mathcal U}_n^x\cap \{T<\infty\}}\\
&\ge \left(1-\exp\left(-\frac{\lambda-\lambda'}{5}\right)\right)^{|\cP|} p^{2|\cP|}\cdot p\cdot  \mathds{1}_{\widetilde{\mathcal U}_n^x\cap \{T<\infty\}},
\end{align*}
where we used that there are at most $2|\cP|$ edges between adjacent squares in the tessellation of $Q_T$ in the definition of $\cE_{n,1}^x$. This proves our claim~\eqref{claim.Enx}.

Now, letting $\widetilde{\mathcal B}_n^x := \bigcap_{i=2}^{W/2} \mathcal B_{n,i}^x$ and observing that it is $\mathcal F_T$-measurable, we can write
$$\mathbb P(\cB_n^x \cap \{T < \infty\})\le \mathbb P(\overline{\mathcal E_n^x}\cap \widetilde{\mathcal B}_n^x\cap \{T<\infty\}) \le \mathbb E\Big[\mathbb P(\overline{\mathcal E_n^x}\mid \mathcal F_T) \cdot \mathds{1}_{\widetilde{\mathcal B}_n^x\cap \{T<\infty\}} \Big] \le (1-\varepsilon) \mathbb P(\widetilde{\mathcal B}_n^x\cap \{T<\infty\}),$$
where for the last inequality we used that the event $\mathcal E_n^x$ has the same distribution conditionally on $\widetilde{\mathcal B}_n^x$ and $\widetilde{\mathcal U}_n^x$, respectively, and we can use claim~\eqref{claim.Enx} with $\mathds{1}_{\widetilde{\mathcal B}_n^x\cap \{T<\infty\}}$ instead of $\mathds{1}_{\widetilde{\mathcal U}_n^x\cap \{T<\infty\}}$; indeed, the fact whether or not there are edges between paths starting from $x$ and some cycle among $\mathcal C_2, \ldots, \mathcal C_{W/2}$ is independent from the state of the edges between these paths and $\mathcal C_1$.
Hence, by~\eqref{eq inclusion} and an immediate induction we deduce that 
$$\mathbb P\big(\mathcal O_n^x\cap \mathcal B_n^x \big)\le \mathbb P(\cB_n^x \cap \{T < \infty\}) \le (1-\varepsilon)^{W/2}.$$ 
Then, by choosing the constant $C_1$ (defined in the beginning of the proof) large enough, we can make the previous bound $o(n^{-1})$, as desired.
\end{proof}

As a consequence of Claim~\ref{cl:1}, we get
$$ \mathbb P\left(\bigcup_{X\in \mathrm{Po}(\lambda)\cap \Lambda_n} \mathcal O_n^X\cap \mathcal B_n^X\right) = o(1).$$ 
Together with \eqref{probaAn} this shows that, with probability going to 1 as $n\to \infty$, any connected component which is not connected to $\widehat{G}_n$ is such that it does not exit the box $\Lambda_n^X$ for any of the vertices $X$ in this component. To conclude, it suffices to observe that, by concentration of Poisson random variables, it is very unlikely that one of these boxes contains much 
more than $(\log n)^2$ points. Indeed, applying again Lemma~\ref{campbell-mecke} gives 
\begin{align*}
& \mathbb P\Big(\bigcup_{X\in \mathrm{Po}(\lambda)\cap \Lambda_n} \big\{|G_n \cap \Lambda_n^X|> 2\lambda (8C_1\log n)^2\big\}\Big)  \le 
\mathbb E\Big[\sum_{X\in  \mathrm{Po}(\lambda)\cap \Lambda_n}  \mathds{1}_{|G_n \cap \Lambda_n^X|> 2\lambda (8C_1\log n)^2}  \Big]  \\ 
=\; 
&\lambda \int_{\Lambda_n} \mathbb P\big(|G_n^x \cap \Lambda_n^x|> 2\lambda (8C_1\log n)^2\big) \, dx  \le \lambda n \mathbb P\big(|G\cap \Lambda(4C_1\log n)|> 2\lambda(8C_1\log n)^2-1\big) \\
=\; 
&o(1),
\end{align*}
where we use Lemma~\ref{Chernoff poisson} for the last equality. This concludes the proof of the proposition.  
\end{proof}

\begin{remark}\label{rem stretched exp decay}
Apart from the size of the second-largest component in $G_n$, with minor modifications the proof also implies a stretched exponential decay of $|\cC_1(\bo)|$ in the infinite-volume limit conditionally on $|\cC_1(\bo)| < \infty$. As above, the key point is to establish an upper bound on the diameter of $\cC_1(\bo)$. By similar arguments one may show that the event $\cE_n$ that `$A(n, 2n) = \Lambda(2n)\setminus \Lambda(n)$ contains $\Theta(n)$ cycles in $G(\lambda',p)$ (with the notation from the previous proof) that are pairwise at distance at least 1 from each other which all participate in the infinite component of $G(\lambda',p)$' is satisfied with probability $1 - \exp(-\Omega_{\lambda, p}(n))$. Thus, partitioning $\mathbb R^2\setminus \Lambda(1)$ into the annuli $(A(2^{k}, 2^{k+1}))_{k\ge 0}$ and applying the argument from the proof of Proposition~\ref{prop:second.upper} to each of them shows that, conditionally on $\cE_{2^k}$,
the event $\bo\xleftrightarrow{G(\lambda, p)} (\mathbb R^2\setminus \Lambda(2^k))$ and simultaneously $|\cC_1(\bo)| < \infty$ holds with probability at most $\exp(-\Omega_{\lambda, p}(2^k))$. Hence, since $\Lambda(2^{\lfloor \log n\rfloor - 2})$ is included in the ball of radius $n/2$ around $\bo$, the probability that the Euclidean diameter of $\cC_1(\bo)$ is at least $n$ is at most
\begin{align*}
\mathbb P(\cC_1(\bo)\cap (\mathbb R^2\setminus \Lambda(2^{\lfloor \log n\rfloor - 2}))\neq \emptyset) 
&\le\; \mathbb P(\cC_1(\bo)\cap (\mathbb R^2\setminus \Lambda(2^{\lfloor \log n\rfloor - 2}))\neq \emptyset\mid \cE_{2^{\lfloor \log n\rfloor - 1}}) + \mathbb P(\overline{\cE_{2^{\lfloor \log n\rfloor - 1}}})\\
&=\; \exp(-\Omega_{\lambda, p}(n)).
\end{align*}
\end{remark}

\subsection{\texorpdfstring{The convergence of $(n^{-1} L_1(G_n))_{n\ge 1}$.}{}}\label{sec:convergence}
First, we recall that the convergence in probability of $(n^{-1} L_1(G_n))_{n\ge 1}$ to $\lambda \theta(\lambda,p)$ is proved in \cite{PenroseDraft}. On the other hand, $L_1(G_n)$ is bounded by the total number of points of 
$\mathrm{Po}(\lambda)$ in $\Lambda_n$ (which is distributed as a Poisson random variable with parameter $\lambda n$), 
so this sequence is also bounded in $L^p$ for all $p\ge 1$ (e.g. as a consequence of Lemma~\ref{Chernoff poisson}). Therefore, the 
convergence to 
$\lambda \theta(\lambda,p)$ holds in fact in $L^p$ for all $p\ge 1$.

\vspace{0.2cm}

We now prove the almost sure convergence by using the results of Section~\ref{sec.second}. Fix $\delta <1/2$, and for any large integer $n$, let $\cT_n$ be a tessellation of $\Lambda_n$ into $k := \lfloor n^{(1-\delta)/2}\rfloor^2$ squares with volume roughly $n^\delta$. Then, consider the event $\mathcal A_n$ that: 
\begin{itemize}
    \item in each square in $\cT_n$, the largest component has a size of order at least $c n^{\delta}$, where $c > 0$ is some sufficiently small constant, and any other component has diameter at most $(\log n)^2$; 
    \item in $\Lambda_n$, the second-largest component has size at most $(\log n)^3$.
\end{itemize}
The proof of Proposition~\ref{prop:second.upper} shows that, if $c$ is sufficiently small, $\cA_n$ holds with probability $1 - o(n^{-2})$. 

Moreover, observe that on the event $\mathcal A_n$, $L_1(G_n)$ is equal to the sum of $k$  independent terms $(X_i)_{i=1}^k$, all distributed as $L_1(G_{n^\delta})$, up to some error term, which is due to the components lying in the region $R$ of all points in $\Lambda_n$ at distance at most $(\log n)^2$ from the boundaries of the squares of $\mathcal T_n$. For all $i\in \{1,\dots,k\}$, let us denote $Y_i = \min\{n^{-\delta} X_i, 2\}$, and $Y = Y_1 + \ldots + Y_k$. Then, since $X_i\ge cn^{\delta}$ on the event $\cA_n$, $\mathbb E Y_i= \Theta(1)$ for any $i\in \{1,\dots,k\}$. Therefore, by Hoeffding's inequality (which is a version of Chernoff's inequality for bounded random variables, see for example~\cite{Hoeffding}), $\Pr(|Y-\mathbb E Y|\ge \frac{k}{\log n}) = \exp(-n^{\Omega(1)})$. At the same time, $Y_i = n^{-\delta} X_i$ on the event that the $i$-th square in the tessellation contains at most $2n^{\delta}$ points, which happens with probability at least $1-\exp(-n^{\Omega(1)})$ (see Lemma~\ref{Chernoff poisson}), and in particular $\mathbb E Y = n^{-\delta}\mathbb E X + o(1)$. As a consequence,
\begin{align*}
\Pr\left(|X-\mathbb E X|\ge \frac{n}{\log n}\right)
&\le\; \Pr\left(|Y-\mathbb E Y|\ge \frac{k}{2\log n}\right) + \sum_{i=1}^k \Pr(Y_i \neq n^{-\delta} X_i)\\
&\le\; \Pr\left(|Y-\mathbb E Y|\ge \frac{k}{2\log n}\right) + k\exp(-\Omega(n^{\delta})) = \exp(-n^{\Omega(1)}).
\end{align*}

\noindent
On the other hand, the event $\cD_n$ that there are at most $n^{1-\delta/2} (\log n)^3$ points in $R$ (which has area $O(n^{1-\delta/2} (\log n)^2)$) holds with probability $1 - o(n^{-2})$.

Finally, together with the fact that 
\begin{equation*}
\frac{1}{n} \mathbb E\left[X\right] \to \lambda \theta(\lambda) \text{ as } n\to \infty,
\end{equation*}
we deduce that,
\begin{equation*}
\mathbb P\left(|L_1(G_n) - \mathbb E L_1(G_n)|\ge \frac{n}{2 \log n}\right)\le \mathbb P\left(\left|X - \mathbb E\left[X\right]\right|\ge \frac{n}{\log n}\right) + \mathbb P(\overline{\mathcal A_n}) + \mathbb P(\overline{\mathcal D_n}) = o(n^{-2}),
\end{equation*}
and the fact that $\sum_{n\ge 1} n^{-2} < \infty$ together with the Borel-Cantelli lemma concludes the proof of the almost sure  convergence.

\paragraph{Acknowledgements.} The authors are grateful to the two anonymous referees for several useful comments and suggestions.

\bibliographystyle{plain}
\bibliography{Refs}
\end{document}